\documentclass[11pt,final]{article}

\usepackage{amsmath, amscd, amsfonts, amssymb, amsthm, tikz, times, dsfont, enumerate}
\usepackage[T1]{fontenc}
\usepackage{graphicx}
\usepackage[title]{appendix}
\usepackage{bbm}
\usepackage{kotex}
\usepackage{regexpatch}
\usepackage{subfig}
\usepackage{hyperref}
\usepackage[in]{fullpage}

\makeatletter
\newcommand*{\rom}[1]{\expandafter\@slowromancap\romannumeral #1@}
\makeatother

\theoremstyle{plain}
\newtheorem{thm}{Theorem}[section]
\newtheorem{lem}[thm]{Lemma}

\newtheorem{prop}[thm]{Proposition}
\newtheorem{defi}[thm]{Definition}

\newtheorem{remark}[thm]{Remark}

\newcommand{\prob}{\mathbb{P}}
\newcommand{\N}{\mathbb{N}}
\newcommand{\R}{\mathbb{R}}
\newcommand{\C}{\mathbb{C}}
\newcommand{\E}{\mathbb{E}}

\newcommand{\eps}{\epsilon}
\newcommand{\indic}{\mathbbm{1}}
\newcommand{\nn}{\nonumber}
\newcommand{\ER}{Erd\H{o}s-R\'{e}nyi}

\newcommand{\ee}{\mathbf{e}}
\newcommand{\vv}{\mathbf{v}}
\newcommand{\uu}{\mathbf{u}}
\newcommand{\ww}{\mathbf{w}}
\newcommand{\xx}{\mathbf{x}}
\newcommand{\yy}{\mathbf{y}}

\newcommand{\var}{\text{Var}}

\newcommand{\FF}{\mathcal{F}}

\newcommand{\indep}{\rotatebox[origin=c]{90}{$\models$}}

\begin{document}

\title{\Large Noise sensitivity of second-top eigenvectors of \ER~graphs\\ and sparse matrices}
\author{Jaehun Lee \footnote{Department of Mathematical Sciences, KAIST, ljhiverson@kaist.ac.kr}}
\date{January 10, 2020}

\maketitle

\abstract{We consider eigenvectors of adjacency matrices of Erd\H{o}s-R\'{e}nyi graphs and study the variation of their directions by resampling the entries randomly. Let $\mathbf{v}$ be the eigenvector associated with the second-largest eigenvalue of the Erd\H{o}s-R\'{e}nyi graphs. After choosing $k$ entries of the given matrix randomly and resampling them, we obtain another eigenvector $\mathbf{w}$ corresponding to the second-largest eigenvalue of the matrix obtained from the resampling procedure. We prove that, in a certain sparsity regime, $\mathbf{w}$ is “almost” orthogonal to $\mathbf{v}$ with high probability if $k\gg N^{5/3}$. On the other hand, if $k\ll q^2 N^{2/3}$, where $q$ is the sparsity parameter, we observe that $\mathbf{v}$ and $\mathbf{w}$ are “almost” collinear. This extends the recent work of Bordenave, Lugosi and Zhivotovskiy to the Erd\H{o}s-R\'{e}nyi model.}
\vspace{10mm}
%
%\textit{MSC:}
%
%\vspace{5mm}
%
%\textit{Keywords:} 
%
%\newpage
%
%\tableofcontents

\section{Introduction}
Let $X=(x_{ij})$ be a symmetric $N\times N$ matrix with positive integer $N$. The matrix $X$ is said to be a {\it Wigner random matrix} if it satisfies the following properties:
\begin{itemize}
	\item[(i)] For $i\le j$, the $x_{ij}$ are independent random variables and centered. The other entries are automatically defined by symmetry.
	\item[(ii)] For some $\sigma>0$, 
	\begin{align*} 
	\E x_{ij}^{2} = \begin{cases}1&\text{for }i<j,\\\sigma^{2}&\text{for }i=j.\end{cases}
	\end{align*}
\end{itemize}
The class of Wigner matrices is one of the most important classes in random matrix theory. The spectrum of a Wigner matrix has been deeply analyzed and many remarkable results have been proved thus far. Eigenvectors of  Wigner matrices have also attracted much attention. The joint distribution of the coordinates, the size of the largest (or smallest) coordinate and the $\ell^{p}$-norm are the main properties of interest.\\
\indent Very recently, Bordenave, Lugosi and Zhivotovskiy investigated another aspect of eigenvectors, especially the {\it top eigenvector}, the unit eigenvector corresponding to the largest eigenvalue \cite{BLZ19}. They studied how the direction of the top-eigenvector varies by {\it resampling} some entries of a given Wigner matrix. In other words, their main interest is the {\it noise sensitivity} of the top eigenvector. For the notion of noise sensitivity, we refer to the seminal work of Benjamini, Kalai, and Schramm \cite{BKS99}.\\
\indent The resampling procedure introduced in \cite{BLZ19} is as follows.
 For a positive integer $k\le\frac{N(N+1)}{2}$, let $S_{k}=\{(i_{1}j_{1}),\cdots, (i_{k}j_{k})\}$ be a random set of $k$ distinct pairs of positive integers (with $i_{m}\le j_{m}$) which is chosen uniformly from the family of all sets of $k$ distinct pairs. In the resampling procedure, a pair $(i_{m}j_{m})$ is used to denote an index of a matrix entry.
\begin{defi}[Resampling procedure]\label{def: resampling}
	 Let $X'=(x_{ij}')$ be an independent copy of $X$. Write the new random matrix $X^{[k]}=(x_{ij}^{[k]})$ generated from the given Wigner matrix $X$, by resampling entries. For $i\le j$, we define $x_{ij}^{[k]}$ according to 
	\begin{align}
	x_{ij}^{[k]} = \begin{cases}
	x_{ij}' & (ij)\in S_{k}, \\
	& \\
	x_{ij}  & (ij)\notin S_{k}.
	\end{cases}
	\end{align}
	The remaining entries of $X^{[k]}$ are determined by symmetry.
\end{defi}
Let $\lambda_{1}\ge\cdots\ge\lambda_{N}$ be the ordered eigenvalues of $X$ and, let $\vv_{1}$ be the top eigenvector of $X$. Similarly, we use the notation $\lambda_{1}^{[k]}\ge\cdots\ge\lambda_{N}^{[k]}$ and $\vv_{1}^{[k]}$ for the ordered eigenvalues and the top eigenvectors of $X^{[k]}$.\\
\indent Here we assume one additional property for the distribution of $x_{ij}$:
\begin{itemize}\label{eq: subexponential decay}
	\item[(iii)] There exists some constant $\vartheta>0$ such that, for all $i\le j$,
	\begin{align}
	\E [\exp(|x_{ij}|^{\vartheta})]\le\vartheta^{-1}.
	\end{align}
	% sub-exponential decay. See DARMT
\end{itemize}
\begin{thm}[{\cite[Theorem 1]{BLZ19}}]\label{thm: BLZ thm1}
	Let $X$ be a Wigner matrix and assume \eqref{eq: subexponential decay}. Let $X^{[k]}$ be the matrix obtained by the resampling procedure in Definition \ref{def: resampling}. If $k\gg N^{5/3}$, then
	\begin{align}
	\E\left\lvert \left\langle \vv_{1},\vv_{1}^{[k]} \right\rangle \right\rvert = o(1)
	\end{align}
\end{thm}
According to Theorem \ref{thm: BLZ thm1}, we can say, roughly, that $\vv_{1}$ is almost orthogonal to $\vv_{1}^{[k]}$ when more than $O(N^{5/3})$ entries are resampled. On the other hand, if $k$ is much smaller than $N^{5/3}$ (precise description below), the behavior of $\vv_{1}^{[k]}$ is completely different. In such a case, it can be shown that $\vv_{1}$ and $\vv_{1}^{[k]}$ lie almost on the same line, i.e., 
\begin{align}\label{eq: another ver of BLZ thm2}
\left|\left\langle \vv_{1},\vv_{1}^{[k]}\right\rangle\right| \sim 1.
\end{align}
Let us assume that $\vv_{1}$ and $\vv_{1}^{[k]}$ has almost the same direction, and consider 
\begin{align}
\lVert \vv_{1}-\vv_{1}^{[k]} \rVert_{\infty}.
\end{align}
It is well known that unit eigenvectors of Wigner matrices {\it delocalize}, i.e. the $\infty$-norm is bounded above by $(\log{N})^{C}N^{-1/2}$ for some $C>0$. Due to this delocalization result, it follows that
\begin{align}\label{eq: bound by delocalization}
\lVert \vv_{1}-\vv_{1}^{[k]} \rVert_{\infty} \le \frac{\log N^{C}}{\sqrt{N}}.
\end{align}
We need a sharper bound than \eqref{eq: bound by delocalization} in order to describe \eqref{eq: another ver of BLZ thm2}. This sharper bound is given in the next theorem.

\begin{figure}[h!]\label{fig: illustration}
	\begin{center}
		\includegraphics[scale=.35]{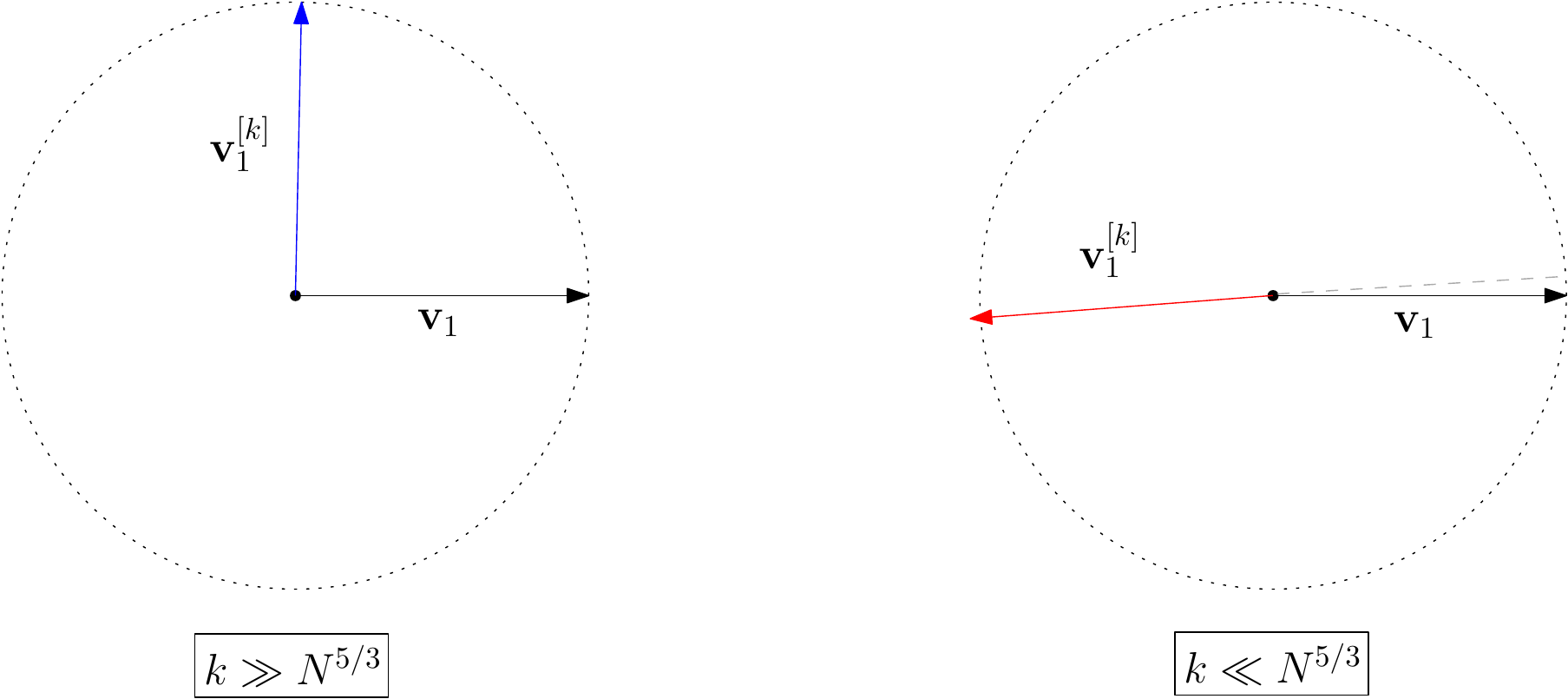}
		\caption{\small An illustration for Theorem \ref{thm: BLZ thm1} and Theorem \ref{thm: BLZ thm2}. On the left, the number $k$ of resampled entries is much larger than $N^{5/3}$. Then, according to Theorem \ref{thm: BLZ thm1}, $\vv_{1}$ are almost orthogonal to $\vv_{1}^{[k]}$. Contrarily, on the right, it is observed that $\vv_{1}$ and $\vv_{1}^{[k]}$ are almost on the same line when $k$ is far less than $N^{5/3}$. In that case, $\vv_{1}$ and $\vv_{1}^{[k]}$ have almost the same direction or almost opposite directions. The latter is illustrated in this figure.}
	\end{center}
\end{figure}

\begin{thm}[{\cite[Theorem 2]{BLZ19}}]\label{thm: BLZ thm2}
	Let $X$ and $X^{[k]}$ be as in Theorem \ref{thm: BLZ thm1}. There exists a constant $C>0$ such that
	\begin{align*}
	\max_{k \le N^{5/3}L^{-C}}\min_{s\in\{\pm 1\}} \sqrt{N}\lVert \vv_{1}-s\vv_{1}^{[k]} \rVert_{\infty}
	\end{align*}
	converges to $0$ in probability where $L$ is a logarithmic control parameter defined by
	\begin{align}\label{eq: logarithmic control parameter}
	L\equiv L_{N}:=(\log N)^{\log\log N},\quad N\ge 3.
	\end{align}
	As a result,
	\begin{align}
	\E\left[ \max_{1\le k \le N^{5/3}L^{-C}}\min_{s\in\{\pm 1\}} \lVert \vv_{1}-s\vv_{1}^{[k]} \rVert_{2} \right] = o(1).
	\end{align}
	where for $N\ge 3$
	\begin{align}
	L\equiv L_{N}:=(\log N)^{\log\log N}.
	\end{align}
\end{thm}

\begin{figure}[h!]\label{fig: simulation for wigner case}
	\begin{center}
		\subfloat[$N=5000$, $k=\lfloor N^{11/6} \rfloor$]{{\includegraphics[width=7cm]{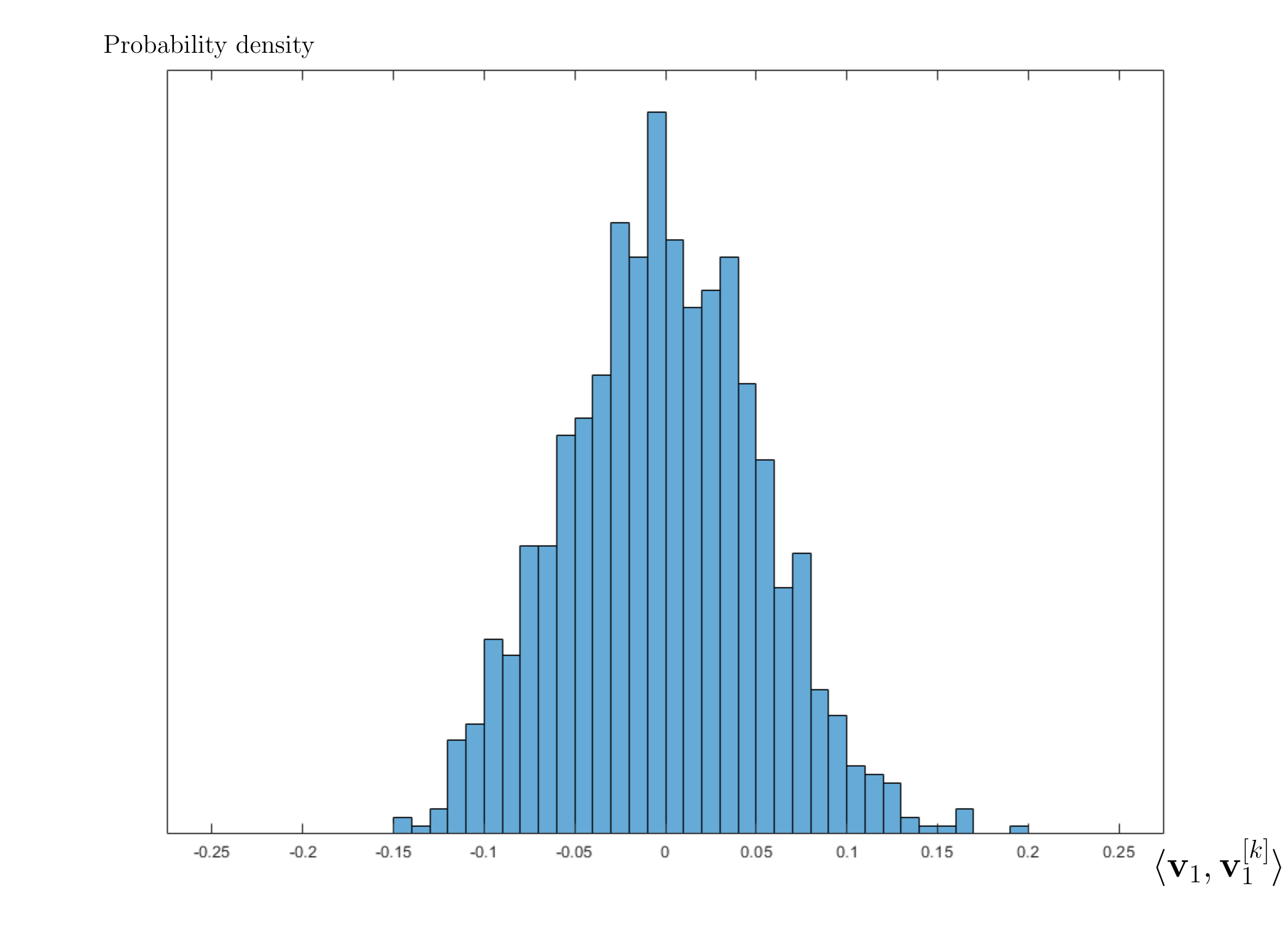}}}
		\qquad
		\subfloat[$N=5000$, $k=\lfloor N^{4/3} \rfloor$]{{\includegraphics[width=7cm]{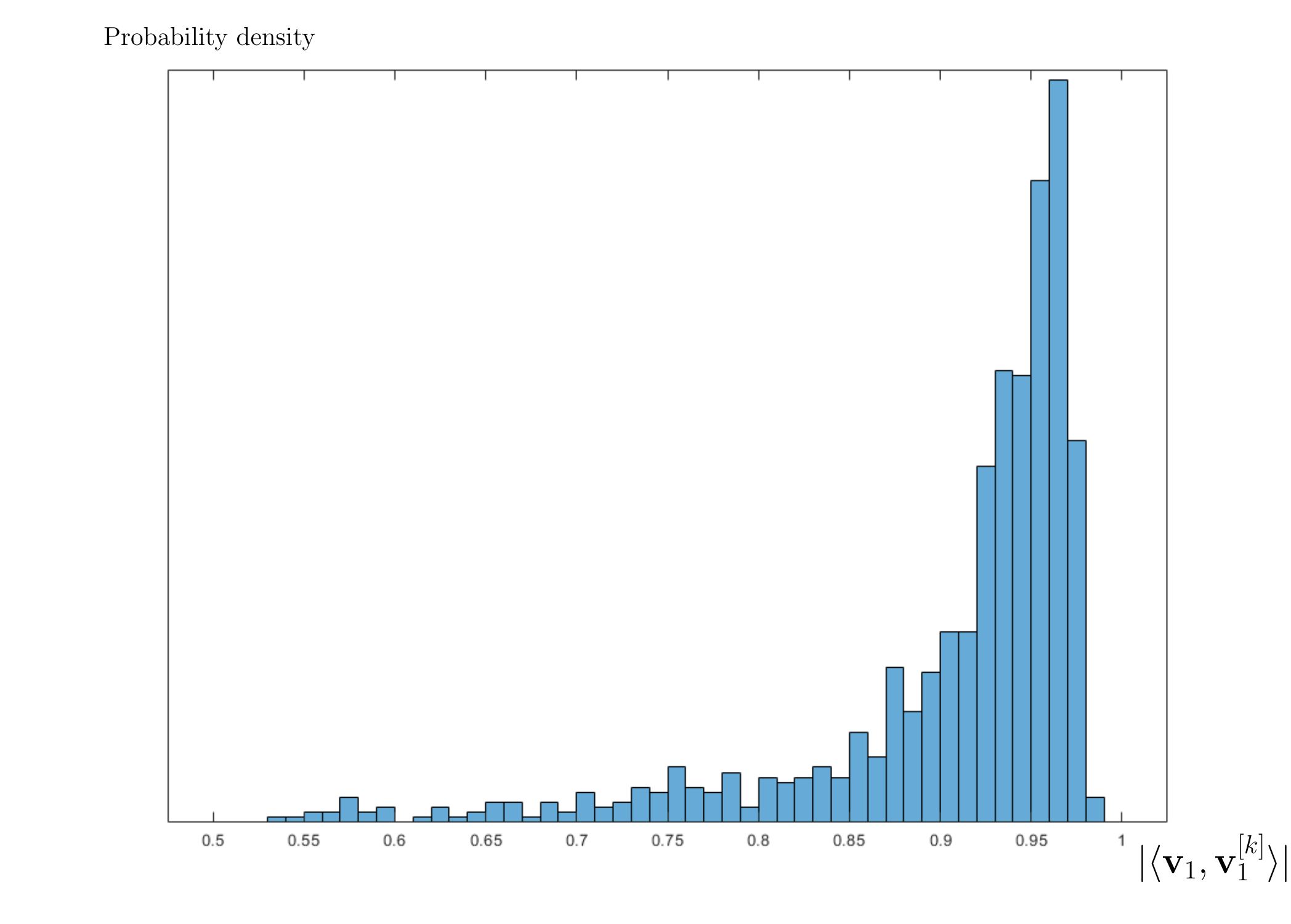}}}
		\caption{\small Simulation for the Wigner random matrices. On the left (a), $k=\lfloor N^{11/6} \rfloor\gg N^{5/3}$ so we observed the realizations of inner product $\langle \vv_{1},\vv_{1}^{[k]} \rangle $ gather around $0$ as expected in Theorem \ref{thm: BLZ thm1}. On the right (b), the contrary case, $k=\lfloor N^{4/3} \rfloor\ll N^{5/3}$, is described. The absolute value of $\langle \vv_{1},\vv_{1}^{[k]} \rangle $ tends to be close to $1$, which is consistent with Theorem \ref{thm: BLZ thm2}.}
	\end{center}
\end{figure}

Combining Theorems \ref{thm: BLZ thm1} and \ref{thm: BLZ thm2}, we can see a phase transition of $\vv_{1}^{[k]}$ around $k\sim N^{5/3}$. Now we raise the following question.
\begin{itemize}
	\item Can we get an analogous result for the class of {\it sparse random matrices}?
\end{itemize}
There are a lot of interesting classes of random matrices besides Wigner matrices. One of the most important is the class of { sparse random matrices}. Roughly speaking, a given random matrix is said to be sparse if it has many zero entries (see Remark \ref{def: sparse RM} for more detailed description). The adjacency matrix of a {\it sparse Erd\H{o}s-R\'{e}nyi graph} is one of the the typical examples. The Erd\H{o}s-R\'{e}nyi graph is a random graph connecting any two distinct nodes by an edge, independently, with probability $p$. Thus, the adjacency matrix of an Erd\H{o}s-R\'{e}nyi graph is a symmetric random matrix and its entries have the Bernoulli distribution with parameter $p$.\\
\indent Let $A\equiv A_{N}$ be the adjacency matrix of the \ER~graph on $N$ vertices with edge density $p$. The matrix $A=(a_{ij})$ is an $N\times N$ symmetric matrix and all entries are independent up to symmetry. %We set $a_{ii}$ to be zero because we do not consider the loop.
The edge density $p$ is replaced by a  {\it sparsity parameter} $q\equiv q(N)$ by setting $q=\sqrt{pN}$. We assume the sparsity parameter $q$ satisfies
\begin{align}
(\log N)^{\log\log N} \ll q \ll \sqrt{N}.
\end{align}
Under a normalization, each entry of the matrix $A=(a_{ij})_{1\le i,j\le N}$ is distributed as follows. If $i\le j$,
\begin{align*}
a_{ij}=\begin{cases}
\zeta/q &  \text{with probability} \; q^{2}/N, \\
& \\
0 & \text{with probability} \; 1-q^{2}/N,
\end{cases}
\end{align*}
where
\begin{align}\label{eq: def of a factor for normalization}
\zeta:=(1-q^{2}/N)^{-1/2}.
\end{align}
From the matrix $A$, we generate the new $N\times N$ matrix $H=(h_{ij})$ by extracting the mean. Each entry $h_{ij}$ is defined according to
\begin{align}
h_{ij}=a_{ij}-\E a_{ij},
\end{align}
and satisfies the moment conditions 
\begin{align}\label{eq: moment condition}
\E |h_{ij}|^{m}\le\frac{1}{Nq^{m-2}} \quad(m\ge 2).
\end{align}

\indent Many outstanding results have been achieved for the \ER~graph \cite{BBK19,BHY17,EKYY12,EKYY13,HKM18}. Especially, it is well-known that the unit eigenvectors of the adjacency matrices of the \ER~graph delocalize under some general conditions. Furthermore, the top eigenvector converges to 
\begin{align}
\frac{1}{\sqrt{N}}(1,\cdots,1)
\end{align}
$N^{-1/2}(1,\cdots,1)\in\C^{N}$ in the $\ell^{2}$-sense with very high probability (see Lemma \ref{lem: delocalization}). This implies that the direction of the top eigenvector is almost deterministic. Therefore the naive answer to the question raised above would be ``No'' for sparse random matrices with nonzero mean.\\
\indent However, it should be noted that the largest eigenvalue of sparse \ER~graphs behaves differently when compared to that of Wigner matrices. In both of these classes of random matrices the eigenvalue distribution converges to the {\it semicircle distribution} (see \eqref{eq: semicircle dist}), but in the case of Wigner matrices, the largest eigenvalue sticks to the edge of the semicircle distribution. In contrast, the largest eigenvalue of the sparse \ER~model stays far away from the support of the semicircle distribution \cite{EKYY13, KS03}. In the sparse \ER~model, the second largest eigenvalue is the largest eigenvalue converging to the edge of the semicircle distribution \cite{BBK17,EKYY13}. 
%Thus, it is natural to correspond the second largest eigenvalue of sparse \ER~graphs to the largest eigenvalue of Wigner random matrices. 
Thus, one might expect that the {\it second top eigenvector}, a unit eigenvector associated with the second largest eigenvalue, shows the desired phase transition behavior.\\
\indent When we consider an eigenvector not associated with the largest eigenvalue, there is one technical difficulty. In \cite{BLZ19}, the authors explicitly use the fact that the top-eigenvector maximizes a quadratic form to prove Theorem \ref{thm: BLZ thm1}, in the following way. Let $X$ and $Y$ be a Wigner random matrix and the one obtained from $X$ by resampling a single entry. Let us say $\vv$ is the top-eigenvector of $X$ and $\ww$ is that of $Y$. Then, by definition, we have
\begin{align}
\langle \ww,(X-Y)\ww \rangle \le \langle \vv,X\vv \rangle - \langle \ww, Y\ww\rangle\le \langle \vv,(X-Y)\vv\rangle
\end{align}
because $\vv$ and $\ww$ maximize $\langle \vv,X\vv \rangle$ and $\langle \ww, Y\ww\rangle$ respectively. Consequently, we need to develop some way to circumvent this issue when considering the second top eigenvector. Actually, it can be done using the fact that the largest eigenvalue is far from the second largest eigenvalue in the sparse \ER~model model.\\
\indent Indeed, we prove that the second top eigenvector of sparse \ER~graphs behaves exactly like the top eigenvector of Wigner matrices under the resampling procedure, when we assume a certain condition on the {\it sparsity parameter}. For the adjacency matrix of an Erd\H{o}s-R\'{e}nyi graph, the sparsity parameter is associated with the expected degree of each node. We shall specify why a certain condition on the sparsity parameter is needed. Interestingly, this is related with the tail bound of the gap between adjacent eigenvalues.\\
\indent It should be noted that research on sparse random matrices have real-world applications. The Erd\H{o}s-R\'{e}nyi graph is a standard model for a random network. In terms of this random graph, resampling an entry corresponds to creating or removing an edge with some probability. Thus, resampling in sparse random matrices can be regarded as  a perturbation of a random network. Moreover, when we analyze data in matrix form, eigenvectors tend to have more information than their associated eigenvalues. Thus, we hope that the above-described phase transition of second-top eigenvectors will find relevance in network theory and the other information sciences.

\section{Results}
Consider the adjacency matrix $A$ of an \ER~graph on $N$ vertices with edge density $p=q^{2}/N$. We consistently assume
\begin{align}
L \ll q \ll \sqrt{N}
\end{align}
where $L$ is the logarithmic control parameter in \eqref{eq: logarithmic control parameter}. Let $A'=(a_{ij}')$ be an independent copy of $A$. We obtain $A^{[k]}=(a_{ij}^{[k]})$ from the adjacency matrix $A$ by following Definition \ref{def: resampling}, the resampling procedure. 
Note that, for $i\le j$, the entry $a_{ij}^{[k]}$ is defined according to 
\begin{align}
a_{ij}^{[k]} = \begin{cases}
a_{ij}' & (ij)\in S_{k}, \\
& \\
a_{ij}  & (ij)\notin S_{k}.
\end{cases}
\end{align}

Let $\lambda_{1}\ge\cdots\ge\lambda_{N}$ be the ordered eigenvalues of $A$ and, let $\vv_{1},\cdots,\vv_{N}$ be associated unit eigenvectors of $A$. Similarly, we use the notation $\lambda_{1}^{[k]}\ge\cdots\ge\lambda_{N}^{[k]}$ and $\vv_{1}^{[k]},\cdots,\vv_{N}^{[k]}$ for the ordered eigenvalues and the associated unit eigenvectors of $A^{[k]}$. Note that Luh and Vu recently showed that sparse random matrices have simple spectrum \cite{LV18}. This implies $\lambda_{1}>\cdots>\lambda_{N}$ a.s. for large $N$. As we already discussed in the introduction, due to Lemma \ref{lem: delocalization}, $\vv_{1}$ and $\vv_{1}^{[k]}$ are almost collinear for any $k$ and large $N$. Thus, we deal with the second top eigenvector.

\begin{thm}[Excessive resampling]\label{thm: main1}
	If $q=N^{b}$ with $b\in(\frac{4}{9},\frac{1}{2})$ and $k=N^{\tau}$ with $\tau\in(\frac{5}{3},2)$, then
	\begin{align*}
	\E\left\lvert \left\langle \vv_{2},\vv_{2}^{[k]} \right\rangle \right\rvert = o(1).
	\end{align*}
\end{thm}
\begin{remark}
	The condition $b\in(\frac{4}{9},\frac{1}{2})$ comes from the tail bound for gaps between the adjacent eigenvalues. It can be improved if we can obtain a sharper estimate for the spacing between two adjacent eigenvalues converging to the edge.  
\end{remark}

\begin{thm}[Small resampling]\label{thm: main2}
	Suppose $q=N^{b}$ with $b\in (\frac{1}{3},\frac{1}{2})$ and $\tau>0$. Assume
	\begin{align}\label{eq: sparsity regime condition}
	\tau < 2b+2/3.
	\end{align}
	Then,
	\begin{align*}
	\max_{k \le N^{\tau}}\min_{s\in\{\pm 1\}} \sqrt{N}\lVert \vv_{2}-s\vv_{2}^{[k]} \rVert_{\infty}
	\end{align*}
	converges to $0$ in probability. As a result,
	\begin{align}
	\lim_{N\to\infty}\E\left\lvert \left\langle \vv_{2},\vv_{2}^{[k]} \right\rangle \right\rvert = 1.
	\end{align}
\end{thm}
\begin{remark}
	The available sparsity regime is determined by $\tau < 2b+\frac{2}{3}$. This is due to Lemma \ref{lem: lem13 in BLZ}, an estimate on the variation of resolvents according to the resampling.
\end{remark}

\begin{figure}[h!]\label{fig: phase diagram}
	\begin{center}
		\includegraphics[scale=.3]{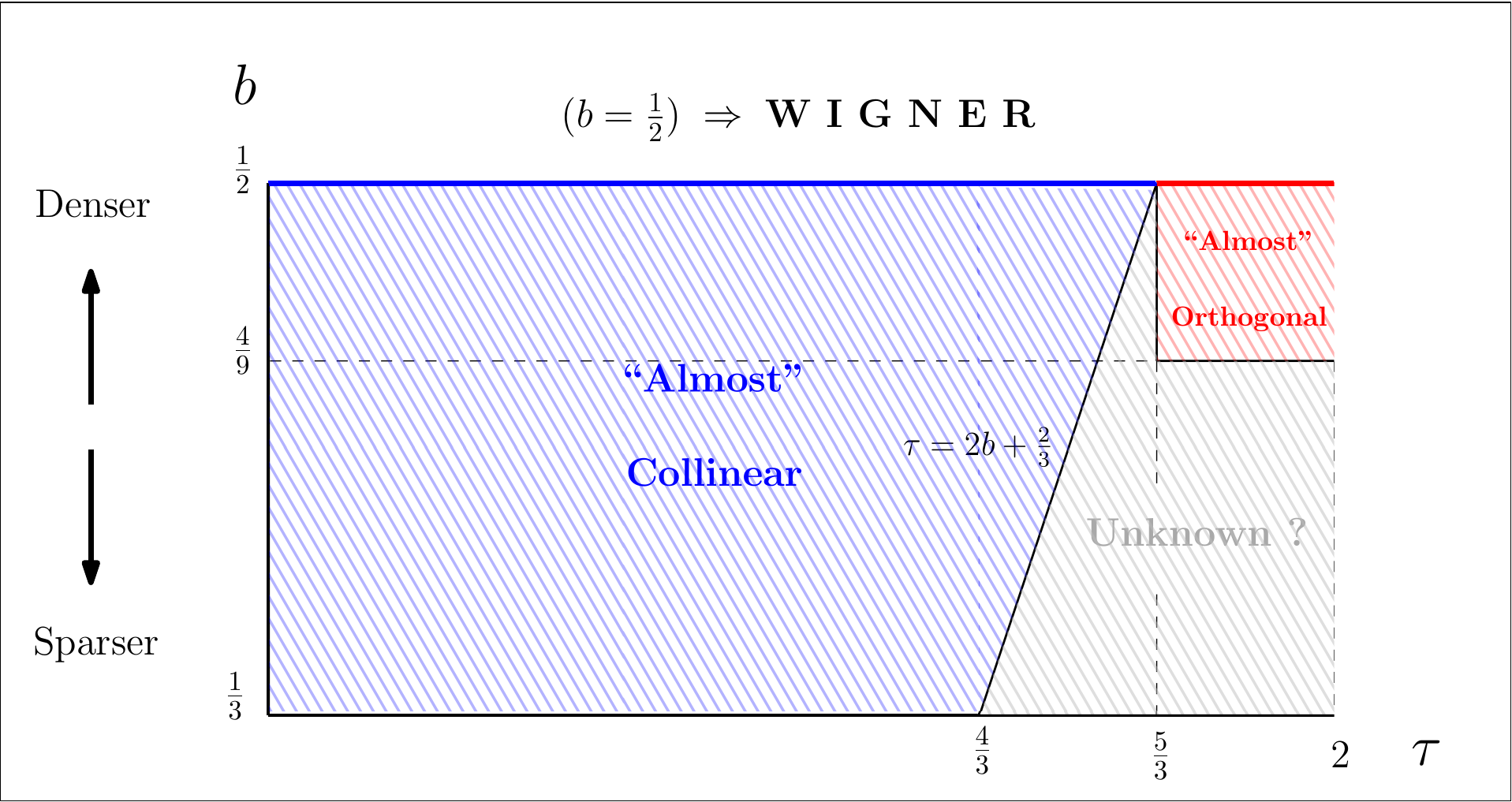}
		\caption{\small Phase diagram for the \ER~model. Note that the line $b=\frac{1}{2}$ corresponds to the Wigner matrix case. The rectangular domain in the upper right corner is for Theorem \ref{thm: main1}. The left quadrangle is the domain for Theorem \ref{thm: main2}. Note that the boundary between the ``almost collinear'' domain and the unknown one is the line $\tau = 2b+2/3$ according to the condition \eqref{eq: sparsity regime condition}.}
	\end{center}
\end{figure}

\begin{figure}[h!]\label{fig: simulation for ER case}
	\begin{center}
		\subfloat[$N=5000$, $k=\lfloor N^{11/6} \rfloor$, $q=N^{17/36}$]{{\includegraphics[width=7cm]{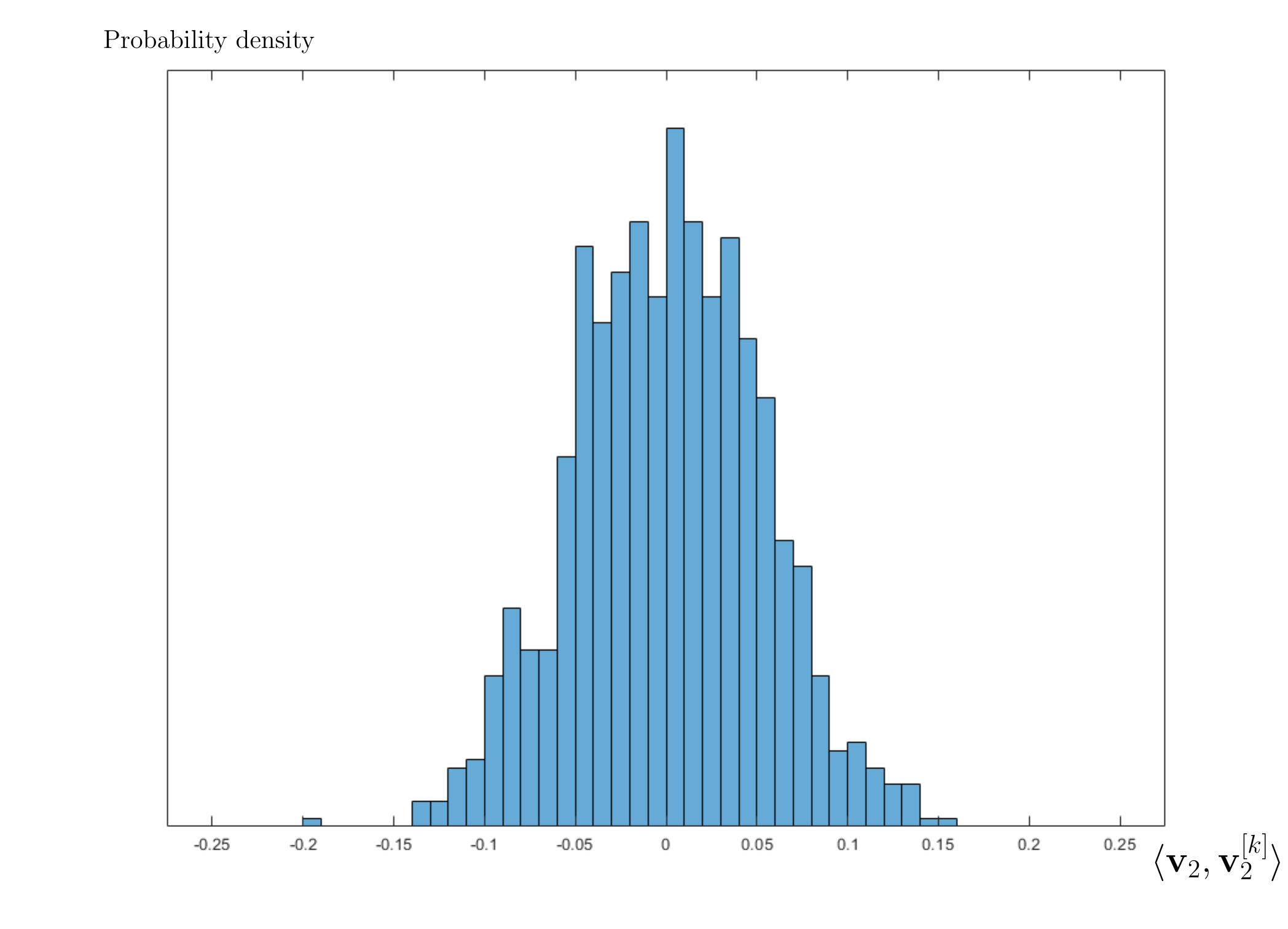}}}
		\qquad
		\subfloat[$N=5000$, $k=\lfloor N^{4/3} \rfloor$, $q=N^{17/36}$]{{\includegraphics[width=7cm]{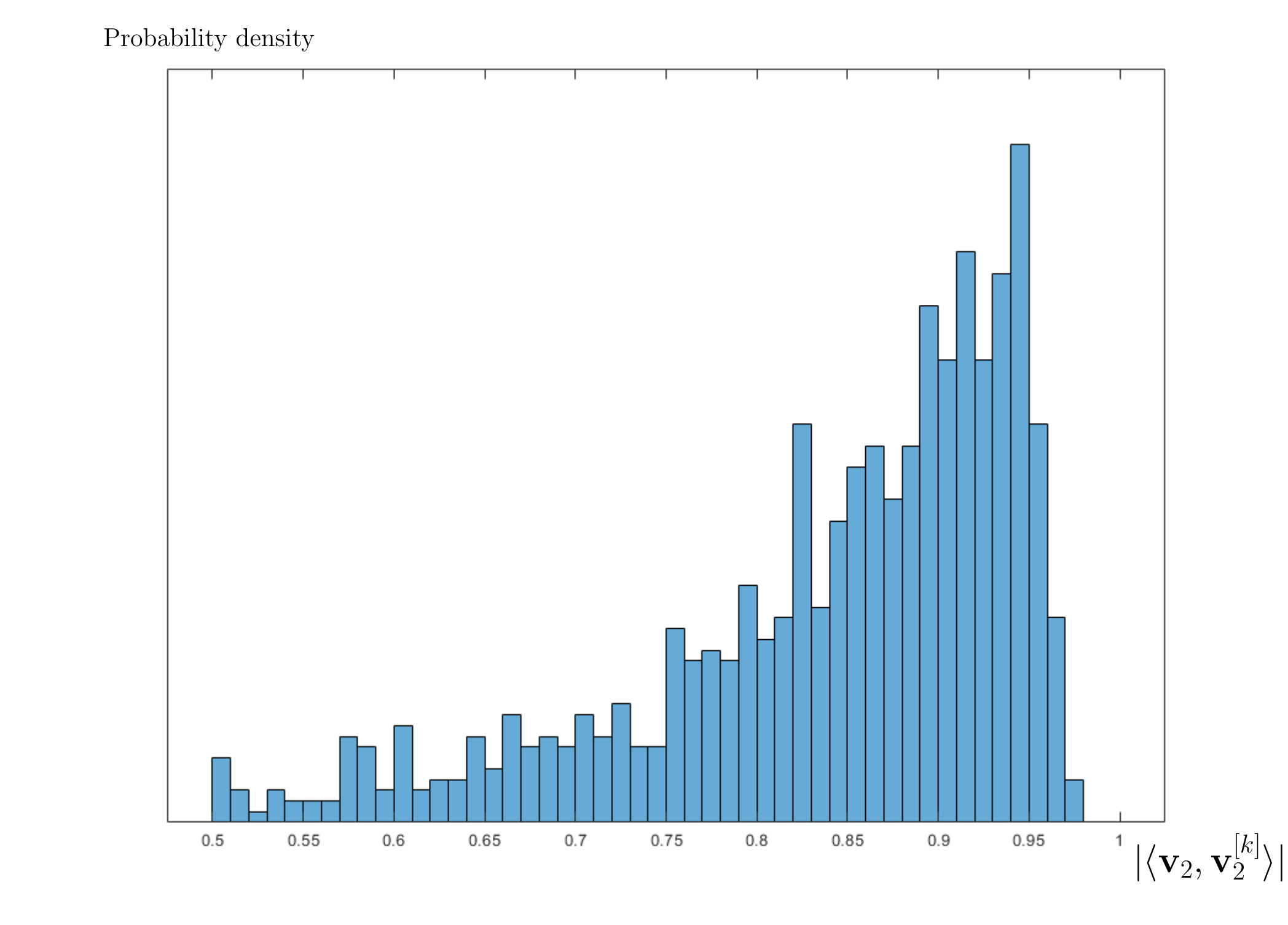}}}
		\caption{\small Simulation for the adjacency matrices of \ER-graphs. On the left (a), the condition $N=5000$, $k=\lfloor N^{11/6} \rfloor$ and $q=N^{17/36}$ satisfies the assumption in Theorem \ref{thm: main1} and it is observed that $\langle \vv_{1},\vv_{1}^{[k]} \rangle$ is close to $0$. On the right (b), we also choose $N$, $k$ and $q$ to satisfy the assumptions in Theorem \ref{thm: main2}  and it is observed that $|\langle \vv_{1},\vv_{1}^{[k]} \rangle|$ is close to $1$.}
	\end{center}
\end{figure}

\begin{remark}[Sparse random matrices]\label{def: sparse RM}
	Theorem \ref{thm: main1} and Theorem \ref{thm: main2} may be extended to the more general class of sparse random matrices. Here we present the model studied in \cite{EKYY12,EKYY13}. Let $\widetilde{H}=(\tilde{h}_{ij})$ be $N\times N$ symmetric random matrices whose entries are real and independent up to the symmetry. We assume that the elements of $\widetilde{H}$ satisfy the moment condition \eqref{eq: moment condition} for all $1\le i,j\le N$ and $3\le m\le L$. We define
	\begin{align}
	\widetilde{A}:=\widetilde{H}+f|\ee\rangle\langle\ee|,
	\end{align}
	where $f\equiv f(N)$ is a deterministic number satisfying
	\begin{align}
	L^{-1}q\le f \le L q,
	\end{align}
	and
	\begin{align}
	\ee \equiv \ee_{N} := \frac{1}{\sqrt{N}}(1,\cdots,1)^{T}.
	\end{align}
	We call this $\widetilde{A}$ a sparse random matrix (with nonzero mean).
\end{remark}

\subsection{Outline}
In the next section, we shall cover some necessary tools used in the proof of the main results. In Section \ref{sec: proof strategy}, we describe the top-level proofs of Theorem \ref{thm: main1} and Theorem \ref{thm: main2}. In Section \ref{sec: monotonicity}, we prove the monotonicity lemma which is an ingredient to the proof of Theorem \ref{thm: main1}. The remaining sections, Section \ref{sec: thm1} and Section \ref{sec: thm2}, are devoted to the details, for Theorem \ref{thm: main1} and Theorem \ref{thm: main2}, respectively.

\subsection{Notation and convention}
We use $C>0$, $D>0$ and $c>0$ as universal constants whose values may change between occurrences. In this paper, $C$ and $D$ are used to denote some constants large enough whereas $c>0$ means a sufficiently small constant. Sometimes we use subscript indices such as $C_{1}$, $\cdots$, $D_{1}$, $\cdots$, $c_{1}$, $\cdots$, whenever we need to denote some fixed large or small constant specifically. Asymptotic notation is used under the assumption $N\to\infty$. For function $f$ and $g$ of parameter $N$, we use the following notation as $N\to\infty$: $f\gtrsim g$ if there exists $C>0$ such that $C\cdot f\ge g$; $f=O(g)$ if $|f|/|g|$ is bounded from above; $f=o(g)$ if $|f|/|g|\to 0$; $f=\Theta(g)$ if $f=O(g)$ and $g=O(f)$. Hilbert-Schmidt norm is denoted by $\lVert \cdot \rVert_{HS}$.

\section{Preliminaries}
In this section, we collect some necessary tools for the proof of main results. For any positive integer $i$, denote $[i]=\{1,\cdots,i\}$. Let $\mathcal{X}_{1},\cdots,\mathcal{X}_{n}$ be i.i.d. random variables taking values in $\R$ and let $f:\R^{n}\mapsto\R$ be a measurable function. Consider the random vector $\mathcal{X}=(\mathcal{X}_{1},\cdots,\mathcal{X}_{n})$ (we shall replace $n$ with $N(N+1)/2$ later). Let $\mathcal{X}'=(\mathcal{X}_{1}',\cdots,\mathcal{X}_{n}')$ be an independent copy of $\mathcal{X}$. We shall use the following notation,
\begin{align}
\mathcal{X}^{(i)}=(\mathcal{X}_{1},\cdots,\mathcal{X}_{i-1},\mathcal{X}_{i}',\mathcal{X}_{i+1},\cdots,\mathcal{X}_{n})\quad\text{and}\quad
\mathcal{X}^{[i]}=(\mathcal{X}_{1}',\cdots,\mathcal{X}_{i}',\mathcal{X}_{i+1},\cdots,\mathcal{X}_{n}),
\end{align}
in particular, $\mathcal{X}^{[0]}=\mathcal{X}$ and $\mathcal{X}^{[n]}=\mathcal{X}'$. More generally, for $\mathcal{I}\subset\N_{+}$, we define $\mathcal{X}^{\mathcal{I}}= (\mathcal{X}^{\mathcal{I}}_{1}, \cdots, \mathcal{X}^{\mathcal{I}}_{n})$ by setting
\begin{align}
 \mathcal{X}^{\mathcal{I}}_{i} = \begin{cases}
  \mathcal{X}_{i} & \text{if }i\notin\mathcal{I}, \\
  \mathcal{X}_{i}'& \text{if }i\in\mathcal{I}.
 \end{cases}
\end{align}
Let $\sigma=(\sigma(1),\cdots,\sigma(n))$ be a random permutation sampled uniformly from the symmetric group $\mathcal{S}_{n}$ and let $\sigma([i])$ denote $\{\sigma(1),\cdots,\sigma(i)\}$. We assume $\sigma$ is independent of $\mathcal{X}$ and $\mathcal{X}'$. Let $j$ be an random variable uniformly distributed on $[n]$ and independent of $\mathcal{X},\mathcal{X}'$ and $\sigma$. Let $\mathcal{X}''$ be an independent copy of $\mathcal{X}$ and be independent of other random variables. Let $\mathcal{X}^{(j)\circ\sigma([i-1])}$ be the vector obtained from $\mathcal{X}^{\sigma([i-1])}$ by replacing $j$-th component of $\mathcal{X}^{\sigma([i-1])}$ with $\mathcal{X}_{j}''$, in particular,
\begin{align}
\mathcal{X}^{(j)} = (\mathcal{X}_{1},\cdots,\mathcal{X}_{j-1},\mathcal{X}_{j}'',\mathcal{X}_{j+1},\cdots,\mathcal{X}_{n}).
\end{align}
For example, suppose $n=5$, and a realization of random elements $\sigma$ and $j$ is given by 
\begin{align}
\sigma=(2,3,1,5,4) \quad \text{and} \quad j=3.
\end{align}
If $i=3$, we have $\sigma([i-1])=\{2,3\}$ and 
\begin{align}
\mathcal{X}^{\sigma[i-1]}=(\mathcal{X}_{1}, \mathcal{X}_{2}', \mathcal{X}_{3}', \mathcal{X}_{4}, \mathcal{X}_{5}) \quad \text{and} \quad
\mathcal{X}^{(j)\circ\sigma[i-1]}=(\mathcal{X}_{1}, \mathcal{X}_{2}', \mathcal{X}_{3}'', \mathcal{X}_{4}, \mathcal{X}_{5}).
\end{align}
\begin{lem}[Variance and noise sensitivity {\cite[Lemma 3]{BLZ19}}]\label{lem: superconcentration}
	For any $i\in[n]$, define $I_{i}$ by
	\begin{align}\label{eq: def of I_i}
	I_{i}:=\E\left[ \left(f(\mathcal{X})-f(\mathcal{X}^{(j)})\right) \left(f(\mathcal{X}^{\sigma([i-1])})-f(\mathcal{X}^{(j)\circ\sigma([i-1])})\right) \right].
	\end{align}
	Then, we have for $i\in[n]$
	\begin{align}\label{eq: superconcentration}
	I_{k}\le\left(\frac{n+1}{n}\right)\left(\frac{2\var \big(f(\mathcal{X})\big)}{k}\right).
	\end{align}
\end{lem}
\begin{lem}[Monotonicity lemma]\label{lem: monotonicity}
	For any $i\in[n]$, define $I_{i}$ as in \eqref{eq: def of I_i}. Then, we have for $i\in[n-1]$
	\begin{align}\label{eq: monotonicity}
	I_{i}\ge I_{i+1}.
	\end{align}
\end{lem}
Lemma \ref{lem: superconcentration} is shown in \cite{BLZ19}. For the reader's convenience, we provide its proof in the appendix, with adding some additional details. In Section \ref{sec: monotonicity}, we establish Lemma \eqref{lem: monotonicity}, the monotonicity of $\{I_{i}\}$. This monotonicity lemma is required to prove Lemma \ref{lem: expectation estimate} later. Next, we summarize some random matrix results.
\begin{defi}[Overwhelming probability]
	Let $\{E_{N}\}_{N\in \N_{+}}$ be a sequence of events. We say $E_{N}$ holds with \textbf{overwhelming probability} if for any $C>0$, there exists a constant $D>0$ such that for all $N\in\N_{+}$
	\begin{align}
	\prob\big((E_{N})^{c}\big)\le DN^{-C}.
	\end{align}
\end{defi}
%\begin{remark}
%	Overwhelming probability is desirable when we use the union bound.
%\end{remark}
Let $\gamma_{i}$ be the classical location of $i$-th eigenvalue which is defined by
\begin{align}\label{eq: def of classical location}
\mu_{\text{sc}}(\gamma_{i},\infty)=\frac{i}{N}
\end{align}
where 
\begin{align}\label{eq: semicircle dist}
\mu_{\text{sc}}:=\frac{1}{2\pi}\sqrt{(4-x^{2})_{+}}
\end{align}
is the semicircle distribution.
\begin{lem}[Eigenvalue location {\cite[Theorem 2.13]{EKYY13}}]\label{lem: eig val location}
	 If $q = N^{b}$ with $b\in(\frac{1}{3},\frac{1}{2})$, then there exists $C>0$ such that we have with overwhelming probability for $i\in\{2,\cdots, N\}$
	\begin{align}
	|\lambda_{i}-\gamma_{i}|\le L^{C}\left(N^{-2/3}\big(\text{min}(i,N-i)\big)^{-1/3}+q^{-2}\right).
	\end{align}
\end{lem}
\begin{remark}
	Using Lemma \ref{lem: eig val location}, we obtain for some $C>0$
	\begin{align}\label{eq: variance estimate}
	\var(\lambda_{2})\le O\left(L^{C}N^{-4/3}\right).
	\end{align}
\end{remark}

\begin{lem}[Location of the largest eigenvalue {\cite[Theorem 6.2]{EKYY13}}]\label{lem: top eig val location}
	We have with overwhelming probability
	\begin{align}
	\lambda_{1} = \zeta q + \frac{1}{\zeta q} + o(1).
	\end{align}
\end{lem}

\begin{lem}[Delocalization of eigenvectors {\cite[Theorem 2.16]{EKYY13}}]\label{lem: delocalization}
	Assume $q=N^{b}$ and $b\in(0,\frac{1}{2})$. There exists a constant $C>0$ such that
	\begin{align}
	\max_{1\le i \le N} \lVert \vv_{i} \rVert_{\infty} \le \frac{(\log{N})^{C}}{\sqrt{N}}
	\end{align}
	holds with overwhelming probability. In particular, for the top eigenvector $\vv_{1}$, we have with overwhelming probability
	\begin{align}\label{eq: top eigen vec fail reason}
	\lVert \vv_{1} - \ee \rVert_{2}\le O(q^{-1})
	\end{align}
	where
	\begin{align}
	\ee = \frac{1}{\sqrt{N}}(1,\cdots,1)^{T}.
	\end{align}
\end{lem}

\begin{lem}[Tail bounds for the gaps between eigenvalues {\cite[Theorem 2.6]{LL19}}]\label{lem: tail bound for gaps}
	Assume $q=N^{b}$ and $b\in(0,\frac{1}{2})$. There exists $c>0$ such that we have for any $\rho>0$
	\begin{align}
	\sup_{1\le i\le N-1} \prob\left( \lambda_{i}-\lambda_{i+1} \le cN^{-1-\rho} \right)=O(N^{-\rho}\log{N}).
	\end{align} 
\end{lem}

\section{Top-level proof of the results}\label{sec: proof strategy}
We adapt the method of proof in \cite{BLZ19} by applying recent results for the sparse \ER~graph model, in order to establish Theorem \ref{thm: main1} and Theorem \ref{thm: main2}.
\subsection{Top-level proof of Theorem \ref{thm: main1}}
For any $1\le i,j\le N$, denote by $B_{(ij)}$ the symmetric matrix obtained from $A$ by replacing the entry $a_{ij}$ and $a_{ji}$ with $a_{ij}''$ and $a_{ji}''$ respectively. We obtain $B^{[k]}_{(ij)}$ from $A^{[k]}$ by the same operation. Denote by $(st)$ a random pair of indices chosen uniformly from $\{(ij): 1\le i\le j\le N
\}$. Note that 
\begin{align}
|\{(ij): 1\le i\le j\le N\}|=N(N+1)/2
\end{align}
Let $\mu_{1}\ge\cdots\ge\mu_{N}$ be the ordered eigenvalues of $B_{(st)}$ and, let $\uu_{1},\cdots,\uu_{N}$ be the associated unit eigenvectors of $B_{(st)}$. Similarly, we define $\mu_{1}^{[k]}\ge\cdots\ge\mu_{N}^{[k]}$ and $\uu_{1}^{[k]},\cdots,\uu_{N}^{[k]}$ for $B^{[k]}_{(st)}$. According to Lemma \ref{lem: superconcentration}, we have
\begin{align}\label{eq: application of superconcentration lem}
\E\left[ (\lambda_{2}-\mu_{2})(\lambda_{2}^{[k]}-\mu_{2}^{[k]}) \right]\le\frac{2\text{Var}(\lambda_{2})}{k}\cdot\frac{N(N+1)+2}{N(N+1)}.
\end{align}
Using the next lemma, we can control $\lambda_{2}-\mu_{2}$ and $\lambda_{2}^{[k]}-\mu_{2}^{[k]}$.
\begin{lem}\label{lem: small perturbation variation}
	Let us write $\vv_{2}=(v_{1},\cdots,v_{N})$ and $\uu_{2}=(u_{1},\cdots,u_{N})$. There exists $C>0$ and $D>0$ such that with overwhelming probability 
	\begin{align}
	Z_{st}u_{s}u_{t} - \frac{CL^{D}}{q^{3}N^{2}}
	\le \lambda_{2}-\mu_{2} 
	\le Z_{st}v_{s}v_{t} + \frac{CL^{D}}{q^{3}N^{2}}
	\end{align}
	where 
	\begin{align}
	Z_{st}:=(a_{st}-a_{st}'')(1+\indic(s\neq t)).
	\end{align}
	Similarly, with overwhelming probability,
	\begin{align}
	Z_{st}^{[k]}u_{s}^{[k]}u_{t}^{[k]} - \frac{CL^{D}}{q^{3}N^{2}}
	\le \lambda_{2}^{[k]}-\mu_{2}^{[k]} 
	\le Z_{st}^{[k]}v_{s}^{[k]}v_{t}^{[k]} + \frac{CL^{D}}{q^{3}N^{2}},
	\end{align}
	where $\vv_{2}^{[k]}=(v_{1}^{[k]},\cdots,v_{N}^{[k]})$, $\uu_{2}^{[k]}=(u_{1}^{[k]},\cdots,u_{N}^{[k]})$ and
	\begin{align}
	Z_{st}^{[k]}:=\begin{cases}
	(a_{st}-a_{st}'')(1+\indic(s\neq t)) &\text{if } (st)\notin S_{k},\\
	(a_{st}'-a_{st}'')(1+\indic(s\neq t)) &\text{if } (st)\in S_{k}.
	\end{cases}
	\end{align}
\end{lem}

We write
\begin{align}
T_{1}&:=(Z_{st}v_{s}v_{t}+\eps_{0})(Z_{st}^{[k]}v_{s}^{[k]}v_{t}^{[k]}+\eps_{0}),\\
T_{2}&:=(Z_{st}v_{s}v_{t}+\eps_{0})(Z_{st}^{[k]}u_{s}^{[k]}u_{t}^{[k]}-\eps_{0}),\\
T_{3}&:=(Z_{st}u_{s}u_{t}-\eps_{0})(Z_{st}^{[k]}v_{s}^{[k]}v_{t}^{[k]}+\eps_{0}),\\
T_{4}&:=(Z_{st}u_{s}u_{t}-\eps_{0})(Z_{st}^{[k]}u_{s}^{[k]}u_{t}^{[k]}-\eps_{0}).
\end{align}
where
\begin{align}
\eps_{0}:=\frac{CL^{D}}{q^{3}N^{2}}.
\end{align}
%$\lambda_{2}-\mu_{2}=\langle \vv_{2},A\vv_{2}\rangle - \langle \uu_{2},B\uu_{2} \rangle$ and $\lambda_{2}^{[k]}-\mu_{2}^{[k]}=\langle \vv_{2}^{[k]},A^{[k]}\vv_{2}^{[k]}\rangle - \langle \uu_{2}^{[k]},B^{[k]}\uu_{2}^{[k]} \rangle$
Let us define the event $\mathcal{E}_{1}$: for some $D>0$
\begin{align}\label{eq: event1 cond1}
|x-\gamma_{2}|\le L^{D}N^{-2/3}\;\;\text{for all }x\in\{\lambda_{2},\mu_{2},\lambda_{2}^{[k]},\mu_{2}^{[k]}\},
\end{align}
\begin{align}\label{eq: event1 cond2}
y=\zeta q+\frac{1}{\zeta q}+o(1)\;\;\text{for all }y\in\{\lambda_{1},\mu_{1},\lambda_{1}^{[k]},\mu_{1}^{[k]}\},
\end{align}
\begin{align}\label{eq: event1 cond3}
\max_{1\le i_{1},i_{2},i_{3},i_{4}\le N}(\lVert \vv_{i_{1}}  \rVert_{\infty},\lVert \uu_{i_{2}} \rVert_{\infty},\lVert  \vv_{i_{3}}^{[k]}\rVert_{\infty},\lVert \uu_{i_{4}}^{[k]} \rVert_{\infty}) \le \frac{L^{D}}{\sqrt{N}},
\end{align}
Recall $\gamma_{2}$ is the second classical location defined by \eqref{eq: def of classical location} and $\zeta=(1-q^{2}/N)^{-1/2}$. It follows from Lemma \ref{lem: small perturbation variation} that on the event $\mathcal{E}_{1}$
\begin{align}\label{eq: event 2 effect}
\min(T_{1},T_{2},T_{3},T_{4})\le (\lambda_{2}-\mu_{2})(\lambda_{2}^{[k]}-\mu_{2}^{[k]})\le \max(T_{1},T_{2},T_{3},T_{4}).
\end{align}
%\begin{defi}
%	Suppose a square symmetric matrix $\mathcal{M}$ is given. Then, for the unit eigenvector associated with the second largest eigenvalue of $\mathcal{M}$, we call it \textbf{the second top eigenvector}.
%\end{defi}
\begin{lem}\label{lem: eigen vec small perturbation}
	Assume $q=N^{b}$and $b\in(\frac{4}{9},\frac{1}{2})$. Let $\uu_{2}^{(ij)}$ be the second top eigenvector of $B_{(ij)}$. Then, there exist $\rho>0$ and $\delta>0$ such that
	\begin{align}
	\prob\left(\max_{1\le i,j\le N}\inf_{s\in\{-1,1\}}\lVert s\vv_{2}-\uu_{2}^{(ij)} \rVert_{\infty}> N^{-\frac{1}{2}-\delta}\right)= O(N^{-\rho}\log{N})
	\end{align}
	and 
	\begin{align}
	2b+\rho>1.
	\end{align}
\end{lem}
\begin{remark}
	Lemma \ref{lem: eigen vec small perturbation} provides us with the available sparsity regime: $b\in(\frac{4}{9},\frac{1}{2})$.
\end{remark}
Next, let $\uu_{(ij)}$ and $\uu_{(ij)}^{[k]}$ be the second top eigenvectors of $B^{(ij)}$ and $B^{[k]}_{(ij)}$. We define the event $\mathcal{E}_{2}$ :
\begin{align}
\max_{1\le i,j \le N}\lVert \vv_{2} - \uu_{(ij)} \rVert_{\infty}\le N^{-\frac{1}{2}-\delta} \quad \text{and}
\max_{1\le i,j \le N}\lVert \vv_{2}^{[k]} - \uu_{(ij)}^{[k]} \rVert_{\infty}\le N^{-\frac{1}{2}-\delta}.
\end{align}
According to Lemma \ref{lem: eigen vec small perturbation} (choosing the $\pm$-phase properly for $\uu_{(ij)}$ and $\uu_{(ij)}^{[k]}$), for some $\delta>0$ and $\rho>0$, we have
$\prob(\mathcal{E}_{2}^{c})= O(N^{-\rho}\log{N})$ and also $q^{2}N^{\rho}\gg L^{D}N$ for any $D>0$ . Set the event $\mathcal{E}:=\mathcal{E}_{1}\cap\mathcal{E}_{2}$. On the event $\mathcal{E}$, 
%by Lemma 3.4.3 in Durrett,
we observe that $v_{s}v_{t}u_{s}^{[k]}u_{t}^{[k]}$, $u_{s}u_{t}v_{s}^{[k]}v_{t}^{[k]}$ and $u_{s}u_{t}u_{s}^{[k]}u_{t}^{[k]}$ can be replaced with
\begin{align}
v_{s}v_{t}v_{s}^{[k]}v_{t}^{[k]}+O\left(\frac{L^{D}}{N^{2+\delta}}\right).
\end{align}
Thus, on the event $\mathcal{E}$, it follows that from Lemma \ref{lem: small perturbation variation}
\begin{align}
(\lambda_{2}-\mu_{2})(\lambda_{2}^{[k]}-\mu_{2}^{[k]})\ge Z_{st}Z_{st}^{[k]}v_{s}v_{t}v_{s}^{[k]}v_{t}^{[k]}+|Z_{st}Z_{st}^{[k]}|O\left(\frac{L^{D}}{N^{2+\delta}}\right)+O\left(\frac{L^{D}}{q^{4}N^{3}}\right).
\end{align}
We shall see
\begin{lem}\label{lem: expectation estimate}
	\begin{align}\label{eq: expectation estimate 1}
	\E\left[ Z_{st}Z_{st}^{[k]}v_{s}v_{t}v_{s}^{[k]}v_{t}^{[k]}\indic_{\mathcal{E}} \right]\gtrsim \frac{1}{N^{3}}\E\left[\sum_{1\le i,j \le n}v_{i}v_{j}v_{i}^{[k]}v_{j}^{[k]}\right]+o(N^{-3})
	\end{align}
	and
	\begin{align}\label{eq: expectation estimate 2}
	\E\left[ (\lambda_{2}-\mu_{2})(\lambda_{2}^{[k]}-\mu_{2}^{[k]})\indic_{\mathcal{E}^{c}} \right]=o(N^{-3}).
	\end{align}
\end{lem}
Since $\E[|Z_{st}Z_{st}^{[k]}|]=O(N^{-1})$, % See note on 7/21-22, use $h_{ij}=a_{ij}-\E a_{ij}$ and $\E h_{ij}^{2} = \frac{1}{N}$. Cauchy-Schwartz.
the main estimate
\begin{align}\label{eq: resulting ineq}
\E\left[ (\lambda_{2}-\mu_{2})(\lambda_{2}^{[k]}-\mu_{2}^{[k]}) \right]\gtrsim \frac{1}{N^{3}}\E\left[\sum_{1\le i,j \le n}v_{i}v_{j}v_{i}^{[k]}v_{j}^{[k]}\right]+o(N^{-3})
\end{align}
follows from \eqref{eq: expectation estimate 1} and \eqref{eq: expectation estimate 2}. Now we are ready to prove the main statement. From \eqref{eq: variance estimate}, \eqref{eq: application of superconcentration lem} and \eqref{eq: resulting ineq}, it is derived that
\begin{align}
\E\left[\sum_{1\le i,j \le n}v_{i}v_{j}v_{i}^{[k]}v_{j}^{[k]}\right] %\le \frac{C N^{3}\text{Var}(\lambda_{2})}{k} + O(1)
\le \frac{L^{D}N^{5/3}}{k}+o(1)
\end{align}
By Jensen's inequality, we have
\begin{align}
\E\left[\sum_{1\le i,j \le n}v_{i}v_{j}v_{i}^{[k]}v_{j}^{[k]}\right] \ge \left(\E|\langle \vv_{2}, \vv_{2}^{[k]} \rangle|\right)^{2}.
\end{align}%= \E\left[\left(\sum_{1\le i\le N} v_{i}v_{i}^{[k]}\right)^{2}\right]
Since $k\gg N^{5/3}L^{D}$, the desired conclusion follows. Lemmas \ref{lem: eigen vec small perturbation}, \ref{lem: small perturbation variation} and \ref{lem: expectation estimate} are proved in Section \ref{sec: thm1}.

\subsection{Top-level proof of Theorem \ref{thm: main2}}
For $z=E+i\eta$ with $\eta>0$ and $E\in\R$, we introduce the resolvent matrices
\begin{align}
R(z)=(A-zI)^{-1},
\end{align}
where $I$ denotes the identity matrix. We denote by $R^{[k]}(z)$ the resolvent of $A^{[k]}$. Theorem \ref{thm: main2} is proved by showing the lemma below. We write $\vv_{2}=(v_{2,1},\cdots,v_{2,N})$ and $\vv^{[k]}_{2}=(v^{[k]}_{2,1},\cdots, v^{[k]}_{2,N})$.

\begin{lem}\label{lem: main lem for thm2}
	Suppose $q=N^{b}$ with $b\in(\frac{1}{3},\frac{1}{2})$ and $k=N^{\tau}$ with $\tau\in(0,2)$. Assume $\tau < 2b+\frac{2}{3}$. For any $\eps>0$, we have with probability at least $1-\eps$
	\begin{align}
	\max_{1\le i,j\le N} N\lvert v_{2,i}v_{2,j} - v_{2,i}^{[k]}v_{2,j}^{[k]} \rvert = O(L^{-2}).
	\end{align}
\end{lem}
This lemma is proved in Section \ref{sec: thm2}. Let us fix $\eps>0$. From now on, we shall work on the event such that the conclusions of Lemma \ref{lem: main lem for thm2} hold and we have for some $C_{1}>0$
\begin{align}
\max\left( \lVert \vv_{2} \rVert_{\infty}, \lVert \vv_{2}^{[k]} \rVert_{\infty} \right)\le \frac{(\log{N})^{C_{1}}}{\sqrt{N}}.
\end{align}
According to Lemma \ref{lem: delocalization} and Lemma \ref{lem: main lem for thm2}, this event is of probability at least $1-2\eps$ for large $N$. For each $i$, choose $s_{i}\in\{\pm 1\}$ so that $\text{sgn}(s_{i}v_{2,i}) = \text{sgn}(v_{2,i}^{[k]}) $. Using Lemma \ref{lem: main lem for thm2} with $i=j$, we have
\begin{align}
\sqrt{N}|s_{i}v_{2,i}-v_{2,i}^{[k]}| = O(L^{-1}).
\end{align}
We want to choose a single $s\in\{\pm 1\}$ such that we have for all $i\in\{1,\cdots,N\}$
\begin{align}\label{eq: conclusion of thm2}
\sqrt{N}|sv_{2,i}-v_{2,i}^{[k]}| = o(1).
\end{align}
If $\sqrt{N}|v_{2,i}| \le L^{-1/3}$, then regardless of the choice of $s$, it follows that
\begin{align}
\sqrt{N}|sv_{2,i}-v_{2,i}^{[k]}| &\le \sqrt{N}\left(|v_{2,i}|+|v_{2,i}^{[k]}|\right) \nn\\
& \le L^{-1/3} + \sqrt{N}\left(|v_{2,i}^{[k]}-s_{i}v_{2,i}|+|v_{2,i}|\right) \nn \\
& = O(L^{-1/3}).
\end{align}
Next, consider the other case $\sqrt{N}|v_{2,i}| > L^{-1/3}$. For all $1\le i,j\le N$, we claim 
\begin{align}\label{eq: main claim for thm2}
(1-s_{i}s_{j})N|v_{2,i}v_{2,j}| = O\big(L^{-1}(\log{N})^{C_{1}}\big).
\end{align}
If $s_{i}=s_{j}$, it is trivial so we assume $s_{i}\neq s_{j}$. For brevity, we only consider the case $s_{i}=1$ and $s_{j}=-1$. Furthermore, it is enough to check the case
\begin{align}
v_{2,i}>0,\quad v_{2,i}^{[k]}>0,\quad v_{2,j}>0 \quad\text{and}\quad v_{2,j}^{[k]}<0.
\end{align}
The other cases can be dealt with similarly. Note that
\begin{align}
\sqrt{N}|v_{2,i}-v_{2,i}^{[k]}|=O(L^{-1}), \quad \sqrt{N}|v_{2,j}+v_{2,j}^{[k]}|=O(L^{-1}),
\end{align}
and
\begin{align}
\max\left(|v_{2,i}|,|v_{2,j}|,|v_{2,i}^{[k]}|,|v_{2,j}^{[k]}|\right)\le \frac{(\log{N})^{C_{1}}}{\sqrt{N}}.
\end{align}
As a result, it follows that
\begin{align}
(1-s_{i}s_{j})N|v_{2,i}v_{2,j}| &= 2Nv_{2,i}v_{2,j} \nn \\
&\le N\left|v_{2,i}v_{2,j} - v_{2,i}^{[k]}v_{2,j}^{[k]}\right| + Nv_{2,i}^{[k]}v_{2,j}^{[k]} + Nv_{2,i}v_{2,j} \nn\\
&\le O(L^{-2}) + Nv_{2,i}^{[k]}v_{2,j}^{[k]} + Nv_{2,i}v_{2,j}
\end{align}
where the last inequality is due to Lemma \ref{lem: main lem for thm2}.
Since
\begin{align}
v_{2,i}v_{2,j} = \left(v_{2,i}-v_{2,i}^{[k]}\right)v_{2,j}+v_{2,i}^{[k]}\left(v_{2,j}+v_{2,j}^{[k]}\right) - v_{2,i}^{[k]}v_{2,j}^{[k]},
\end{align}
we establish the claim \eqref{eq: main claim for thm2}.\\
\indent Let us define a set $J:=\{1\le i\le N : \sqrt{N}|v_{2,i}| > L^{-1/3}\}$. If $J=\emptyset$, there is nothing to prove so we suppose $J\neq\emptyset$. 
Assuming this, we find that $s_{i}=s_{j}$ for any $i,j\in J$. Otherwise, we have
\begin{align}
(1-s_{i}s_{j})N|v_{2,i}v_{2,j}|=2N|v_{2,i}v_{2,j}|\ge 2L^{-2/3},
\end{align}
which contradicts \eqref{eq: main claim for thm2}. Thus, we obtain \eqref{eq: conclusion of thm2} by choosing $s=s_{i}$ for some $i\in J$ (such a choice is well-defined because $s_{i}$ is the same for all $i\in J$).\\

\section{Proof of the monotonicity lemma} \label{sec: monotonicity}
First of all, we split $I_{i}$ into two parts.
\begin{align}
I_{i} &= \E\left[ \left(f(\mathcal{X})-f(\mathcal{X}^{(j)})\right) \left(f(\mathcal{X}^{\sigma([i-1])})-f(\mathcal{X}^{(j)\circ\sigma([i-1])})\right) \right] \nn\\
& = \E\left[ \left(f(\mathcal{X})-f(\mathcal{X}^{(j)})\right) \left(f(\mathcal{X}^{\sigma([i-1])})-f(\mathcal{X}^{(j)\circ\sigma([i-1])})\right)\; ; \; j\in\sigma([i-1]) \right] \nn \\
& \quad + \E\left[ \left(f(\mathcal{X})-f(\mathcal{X}^{(j)})\right) \left(f(\mathcal{X}^{\sigma([i-1])})-f(\mathcal{X}^{(j)\circ\sigma([i-1])})\right)\; ; \; j\notin\sigma([i-1]) \right]
\end{align}
%where the outer expectation is with respect to $\sigma$,.and the inner expectation is conditioned on $j$ and $\sigma$.
Let $\E_{\sigma, j}(\cdot)$ be the conditional expectation with respect to $\sigma$ and $j$. We observe that
\begin{align}\label{eq: decomposition of expec}
\E&\left[ \left(f(\mathcal{X})-f(\mathcal{X}^{(j)})\right) \left(f(\mathcal{X}^{\sigma([i-1])})-f(\mathcal{X}^{(j)\circ\sigma([i-1])})\right)\; ; \; j\notin\sigma([i-1]) \right] \nn \\
&\;\; = \frac{1}{n!}\sum_{\substack{\sigma\in S_{N} \\ j\notin \sigma([i-1])} } \E_{\sigma, j}\left[ \left(f(\mathcal{X})-f(\mathcal{X}^{(j)})\right) \left(f(\mathcal{X}^{\sigma([i-1])})-f(\mathcal{X}^{(j)\circ\sigma([i-1])})\right)\right]
\left(\frac{1}{n}\right) \nn \\
&\;\; = \E\left[ \left( f(\mathcal{X}) - f(\mathcal{X}^{\sigma(i)}) \right) \left( f(\mathcal{X}^{\sigma([i-1])}) - f(\mathcal{X}^{\sigma([i])}) \right) \right]  \left(\frac{n-(i-1)}{n}\right).
\end{align}
According to \cite[Lemma 2]{BLZ19}, we have for all $i\in[n-1]$
\begin{align}
\E&\left[ \left( f(\mathcal{X}) - f(\mathcal{X}^{\sigma(i)}) \right) \left( f(\mathcal{X}^{\sigma([i-1])}) - f(\mathcal{X}^{\sigma([i])}) \right) \right] \nn \\
&\quad \ge \E\left[ \left( f(\mathcal{X}) - f(\mathcal{X}^{\sigma(i+1)}) \right) \left( f(\mathcal{X}^{\sigma([i])}) - f(\mathcal{X}^{\sigma([i+1])}) \right) \right].
\end{align}
Thus, it is enough to show that 
\begin{align}
\E&\left[ \left(f(\mathcal{X})-f(\mathcal{X}^{(j)})\right) \left(f(\mathcal{X}^{\sigma([i-1])})-f(\mathcal{X}^{(j)\circ\sigma([i-1])})\right)\; ; \; j\in\sigma([i-1]) \right] \nn \\
&\quad \ge \E \left[ \left(f(\mathcal{X})-f(\mathcal{X}^{(j)})\right) \left(f(\mathcal{X}^{\sigma([i])})-f(\mathcal{X}^{(j)\circ\sigma([i])})\right)\; ; \; j\in\sigma([i]) \right].
\end{align}
Fix $\sigma$ and take $j\in\sigma([i-1])$. We claim that
\begin{align}
\E_{\sigma, j}&\left[ \left(f(\mathcal{X})-f(\mathcal{X}^{(j)})\right) \left(f(\mathcal{X}^{\sigma([i-1])})-f(\mathcal{X}^{(j)\circ\sigma([i-1])})\right) \right] \nn \\
&\quad \ge \E_{\sigma, j} \left[ \left(f(\mathcal{X})-f(\mathcal{X}^{(j)})\right) \left(f(\mathcal{X}^{\sigma([i])})-f(\mathcal{X}^{(j)\circ\sigma([i])})\right)\right].
\end{align}For brevity, we assume $\sigma$ is the identity, i.e., $\sigma(k)=k$ for all $k\in[n]$, and $j=1$. We want to show
\begin{align}
\E_{\sigma, j}&\left[ \left(f(\mathcal{X})-f(\mathcal{X}^{(1)})\right) \left(f(\mathcal{X}^{[i-1]})-f(\mathcal{X}^{(1)\circ[i-1]})\right) \right] \nn \\
&\quad \ge \E_{\sigma, j}\left[ \left(f(\mathcal{X})-f(\mathcal{X}^{(1)})\right) \left(f(\mathcal{X}^{[i]})-f(\mathcal{X}^{(1)\circ[i]})\right) \right].
\end{align}
Note that $\mathcal{X}^{(1)}=(\mathcal{X}_{1}'',\mathcal{X}_{2},\cdots,\mathcal{X}_{n})$. Let us use the following notation:
\begin{align}
J&:=(\mathcal{X}_{2},\cdots,\mathcal{X}_{i-1}), \nn\\
J'&:=(\mathcal{X}_{2}',\cdots,\mathcal{X}_{i-1}'), \nn\\
K&:=(\mathcal{X}_{i+1},\cdots,\mathcal{X}_{n}).
\end{align}
Since $(\mathcal{X}_{1},J)$ and $(\mathcal{X}_{1}',J')$ are i.i.d., we have
\begin{align}
\E_{\sigma, j}&\left[ \left(f(\mathcal{X}_{1},J,\mathcal{X}_{i},K)-f(\mathcal{X}_{1}'',J,\mathcal{X}_{i},K)\right) \left(f(\mathcal{X}_{1}',J',\mathcal{X}_{i},K)-f(\mathcal{X}_{1}'',J',\mathcal{X}_{i},K)\right) | \mathcal{X}_{1}'',\mathcal{X}_{i},K\right] \nn \\
&= \left( \E_{\sigma, j}\left[ \left(f(\mathcal{X}_{1},J,\mathcal{X}_{i},K)-f(\mathcal{X}_{1}'',J,\mathcal{X}_{i},K)\right) | \mathcal{X}_{1}'',\mathcal{X}_{i},K\right] \right)^{2}.
\end{align}
Similarly, it follows that
\begin{align}
\E_{\sigma, j}&\left[ \left(f(\mathcal{X}_{1},J,\mathcal{X}_{i},K)-f(\mathcal{X}_{1}'',J,\mathcal{X}_{i},K)\right) \left(f(\mathcal{X}_{1}',J',\mathcal{X}_{i}',K)-f(\mathcal{X}_{1}'',J',\mathcal{X}_{i}',K)\right) | \mathcal{X}_{1}'',K\right] \nn \\
&= \left( \E_{\sigma, j} \left[ \left(f(\mathcal{X}_{1},J,\mathcal{X}_{i},K)-f(\mathcal{X}_{1}'',J,\mathcal{X}_{i},K)\right) | \mathcal{X}_{1}'',K \right] \right)^{2} \nn \\
&= \left( \E_{\sigma, j} \left[ \E_{\sigma, j} \left[ \left(f(\mathcal{X}_{1},J,\mathcal{X}_{i},K)-f(\mathcal{X}_{1}'',J,\mathcal{X}_{i},K)\right) | \mathcal{X}_{1}'',\mathcal{X}_{i},K \right]  | \mathcal{X}_{1}'',K \right] \right)^{2} \nn \\
&\le \E_{\sigma, j} \left[ \left( \E_{\sigma, j} \left[ \left(f(\mathcal{X}_{1},J,\mathcal{X}_{i},K)-f(\mathcal{X}_{1}'',J,\mathcal{X}_{i},K)\right) | \mathcal{X}_{1}'',\mathcal{X}_{i},K \right]\right)^{2}  | \mathcal{X}_{1}'',K \right].
\end{align}
Therefore, we obtain that
\begin{align}
\E_{\sigma, j}&\left[ \left(f(\mathcal{X})-f(\mathcal{X}^{(1)})\right) \left(f(\mathcal{X}^{[i]})-f(\mathcal{X}^{(1)\circ[i]})\right) \right] \nn \\
&\le \E_{\sigma, j}\left[  \E_{\sigma, j} \left[ \left( \E_{\sigma, j} \left[ \left(f(\mathcal{X}_{1},J,\mathcal{X}_{i},K)-f(\mathcal{X}_{1}'',J,\mathcal{X}_{i},K)\right) | \mathcal{X}_{1}'',\mathcal{X}_{i},K \right]\right)^{2}  | \mathcal{X}_{1}'',K \right] \right] \nn \\
%& = \E_{\sigma, j}\left[  \E_{\sigma, j} \left[ \E_{\sigma, j}\left[ \left(f(\mathcal{X})-f(\mathcal{X}^{(1)})\right) \left(f(\mathcal{X}^{[i-1]})-f(\mathcal{X}^{(1)\circ[i-1]})\right) | \mathcal{X}_{1}'',\mathcal{X}_{i},K \right]  | \mathcal{X}_{1}'',K \right] \right] \nn \\
& = \E_{\sigma, j}\left[ \left(f(\mathcal{X})-f(\mathcal{X}^{(j)})\right) \left(f(\mathcal{X}^{\sigma([i-1])})-f(\mathcal{X}^{(j)\circ\sigma([i-1])})\right) \right] \nn.
\end{align}
\begin{remark}
	When we prove Lemma \ref{lem: expectation estimate} later, this monotonicity lemma, the inequality \eqref{eq: monotonicity}, is used to deal with the case $k\ge N^{2}L^{-D}$. In other words, due to the monotonicity, it is enough to consider the case $k< N^{2}L^{-D}$
\end{remark}

\section{Excessive resampling}\label{sec: thm1}

\subsection{Proof of Lemma \ref{lem: small perturbation variation}}
By spectral theorem, we have
\begin{align}
\langle \uu_{2},A\uu_{2}\rangle = \lambda_{1}|\langle \uu_{2},\vv_{1}\rangle|^{2}+\sum_{i=2}^{N}\lambda_{i}|\langle \uu_{2},\vv_{i}\rangle|^{2}
%&\le \lambda_{1}|\langle \uu_{2},\vv_{1}\rangle|^{2}+\sum_{i=2}^{N}\lambda_{2}|\langle \uu_{2},\vv_{i}\rangle|^{2} \\
%&= \lambda_{1}|\langle \uu_{2},\vv_{1}\rangle|^{2}+\lambda_{2}(1-|\langle \uu_{2},\vv_{1}\rangle|^{2})\\
%&= (\lambda_{1}-\lambda_{2})|\langle \uu_{2},\vv_{1}\rangle|^{2}+\lambda_{2} \\
\le (\lambda_{1}-\lambda_{2})|\langle \uu_{2},\vv_{1}\rangle|^{2}+\langle \vv_{2},A\vv_{2} \rangle.
\end{align}
Recall the event $\mathcal{E}_{1}$ defined by \eqref{eq: event1 cond1}, \eqref{eq: event1 cond2} and \eqref{eq: event1 cond3}. Note that $\mathcal{E}_{1}$ holds with overwhelming probability due to  Lemma \ref{lem: eig val location}, \ref{lem: top eig val location} and \ref{lem: delocalization}. We write 
\begin{align}
\vv_{1}=\alpha\uu_{2}+\beta\xx
\end{align}
where $\xx\in\text{span}(\uu_{1},\uu_{3},\cdots,\uu_{N})$ and $\lVert \xx\rVert=1$.
Since
\begin{align}
B\vv_{1}=A\vv_{1}+(B-A)\vv_{1}=\lambda_{1}\vv_{1}+(B-A)\vv_{1}
\end{align}% =\mu_{1}\vv_{1}+(\lambda_{1}-\mu_{1})\vv_{1}+(B-A)\vv_{1}
and also
\begin{align}
B\vv_{1}=\alpha \mu_{2} \uu_{2} + \beta B \xx,
\end{align}
it follows that
\begin{align}
\lambda_{1}\vv_{1} = \alpha \mu_{2} \uu_{2} + \beta B \xx + (A-B)\vv_{1}.
\end{align}
Then,
\begin{align}
\lambda_{1}\alpha=\lambda_{1}\langle \uu_{2},\vv_{1}\rangle = \langle \uu_{2},\lambda_{1}\vv_{1}\rangle
=\mu_{2}\alpha+\langle \uu_{2},(A-B)\vv_{1} \rangle.
\end{align}
Consequently, on the event $\mathcal{E}_{1}$, it follows that
\begin{align}
|(\lambda_{1}-\mu_{2})\alpha| = |\langle \uu_{2},(A-B)\vv_{1} \rangle|\le \frac{CL^{D}}{qN}. 
\end{align}
Note that $\lambda_{1}\sim \zeta q+(\zeta q)^{-1}$ and $\mu_{2}\le C$ on the event $\mathcal{E}_{1}$. Finally, on the event $\mathcal{E}_{1}$, we obtain
\begin{align}\label{eq: size of alpha}
|\alpha|\le\frac{CL^{D}}{q^{2}N},
\end{align}
which implies 
\begin{align}\label{eq: compare lambda_2 mu_2}
\langle \uu_{2},A\uu_{2}\rangle \le \frac{CL^{D}}{q^{3}N^{2}}+\langle \vv_{2},A\vv_{2} \rangle.
\end{align}
Similarly, we have
\begin{align}
\langle \vv_{2},B\vv_{2}\rangle \le \frac{CL^{D}}{q^{3}N^{2}}+\langle \uu_{2},B\uu_{2} \rangle.
\end{align}
As a result, it follows that on the event $\mathcal{E}_{1}$
\begin{align}
\langle \uu_{2},(A-B)\uu_{2} \rangle - \frac{CL^{D}}{q^{3}N^{2}}
\le \lambda_{2}-\mu_{2} 
\le \langle \vv_{2},(A-B)\vv_{2}\rangle + \frac{CL^{D}}{q^{3}N^{2}}.
\end{align}%  \langle \vv_{2},A\vv_{2}\rangle - \langle \uu_{2},B\uu_{2} \rangle 
Using the same argument, we observe that on the event $\mathcal{E}_{1}$
\begin{align}
\langle \uu_{2}^{[k]},(A^{[k]}-B^{[k]})\uu_{2}^{[k]} \rangle - \frac{CL^{D}}{q^{3}N^{2}}
\le \lambda_{2}^{[k]}-\mu_{2}^{[k]}
\le \langle \vv_{2}^{[k]},(A^{[k]}-B^{[k]})\vv_{2}^{[k]}\rangle + \frac{CL^{D}}{q^{3}N^{2}}.
\end{align}% \langle \vv_{2}^{[k]},A^{[k]}\vv_{2}^{[k]}\rangle - \langle \uu_{2}^{[k]},B^{[k]}\uu_{2}^{[k]} \rangle

\subsection{Proof of Lemma \ref{lem: eigen vec small perturbation}}
Let $\mu^{(ij)}_{1}\ge\cdots\ge\mu^{(ij)}_{N}$ be the ordered eigenvalues of $B_{(ij)}$ and, let $\uu_{1}^{(ij)},\cdots,\uu_{N}^{(ij)}$ be the associated unit eigenvectors of $B_{(ij)}$. According to \eqref{eq: compare lambda_2 mu_2} and Lemma \ref{lem: delocalization}, we have with overwhelming probability
\begin{align}
\lambda_{2} & \ge \langle \uu_{2}^{(ij)},A\uu_{2}^{(ij)}\rangle - \frac{L^{D}}{q^{3}N^{2}} \\
& = \mu^{(ij)}_{2}+\langle\uu_{2}^{(ij)},(A-B_{(ij)})\uu_{2}^{(ij)}\rangle - \frac{L^{D}}{q^{3}N^{2}} \\
&\ge \mu^{(ij)}_{2}-2(|a_{ij}|+|a_{ij}''|)\lVert \uu_{2}^{(ij)} \rVert_{\infty}^{2} - \frac{L^{D}}{q^{3}N^{2}} \\
&\ge \mu^{(ij)}_{2}-\frac{L^{D}}{qN} %- \frac{CL^{D}}{fq^{2}N^{2}}.
\end{align}
Reversing the role of $A$ and $B^{(ij)}$, we also have with overwhelming probability
\begin{align}
\mu^{(ij)}_{2} \ge \lambda_{2}-\frac{L^{D}}{qN} %- \frac{CL^{D}}{fq^{2}N^{2}}.
\end{align}
Thus, it follows that with overwhelming probability
\begin{align}\label{eq: eigvenvalue perturbation}
\max_{1\le i\le j\le N} |\lambda_{2}-\mu^{(ij)}_{2}| \le \frac{L^{D}}{qN}.
\end{align}
Applying Lemma \ref{lem: tail bound for gaps}, we have for any $\rho>0$ 
\begin{align}\label{eq: gap prob}
\lambda_{2}-\lambda_{3} > c N^{-1-\rho}
\end{align}
with probability $1-O(N^{-\rho}\log{N})$. Let $N':=\lfloor N^{\theta}\rfloor$ with $0<\theta<1$. We write
\begin{align}
\uu_{2}^{(ij)}=\alpha \vv_{1} + \beta \vv_{2} + \gamma_{x} \xx + \gamma_{y} \yy
\end{align}
where $\xx\in\text{span}(\vv_{3},\cdots,\vv_{N'})$, $\yy\in\text{span}(\vv_{N'+1},\cdots,\vv_{N})$ and $\lVert \xx \rVert = \lVert \yy \rVert = 1$.
%i.e.,
%\begin{align}
%\xx&=c_{3}\vv_{3}+\cdots+c_{N'}\vv_{N'} \quad (c_{3}^{2}+\cdots+c_{N'}^{2}=1),\\
%\yy&=c_{N'+1}\vv_{N'+1}+\cdots+c_{N}\vv_{N} \quad (c_{N'+1}^{2}+\cdots+c_{N}^{2}=1).
%\end{align}
Since
\begin{align}
A\uu_{2}^{(ij)} = \alpha\lambda_{1}\vv_{1}+\beta\lambda_{2}\vv_{2}+\gamma_{x}A\xx+\gamma_{y}A\yy
\end{align}
and
\begin{align}
A\uu_{2}^{(ij)} = \lambda_{2}\uu_{2}^{(ij)}+(\mu^{(ij)}_{2}-\lambda_{2})\uu_{2}^{(ij)}+(A-B)\uu_{2}^{(ij)},
\end{align}
we can observe
\begin{align}
\lambda_{2}\uu_{2}^{(ij)} = \alpha\lambda_{1}\vv_{1}+\beta\lambda_{2}\vv_{2}+\gamma_{x}A\xx+\gamma_{y}A\yy +(\lambda_{2}-\mu^{(ij)}_{2})\uu_{2}^{(ij)}+(B-A)\uu_{2}^{(ij)}.
\end{align}
Next, it follows that with probability $1-O(N^{-\rho}\log{N})$
\begin{align}
\lambda_{2}\gamma_{x} &= \lambda_{2}\langle \xx,\uu_{2}^{(ij)}\rangle \nn\\
&= \langle \xx,\lambda_{2}\uu_{2}^{(ij)}\rangle \nn\\
& = \gamma_{x}\langle \xx, A\xx\rangle + (\lambda_{2}-\mu^{(ij)}_{2})\gamma_{x}+\langle \xx, (B-A)\uu_{2}^{(ij)}\rangle \nn\\
& < (\lambda_{2}- c N^{-1-\rho})\gamma_{x} + \frac{L^{D}}{qN} + \lVert \xx \rVert_{\infty} \cdot \frac{C}{q}\cdot\frac{L^{D}}{\sqrt{N}}
\end{align}
where we use \eqref{eq: eigvenvalue perturbation}, \eqref{eq: gap prob} and Lemma \ref{lem: delocalization}. Note that
\begin{align}
\lVert \xx \rVert_{\infty}\le \sum_{m=3}^{N'}|c_{m}|\lVert \vv_{m}\rVert_{\infty} \le \frac{L^{D}}{\sqrt{N}}\sum_{m=3}^{N'}|c_{m}|\le\frac{L^{D}\sqrt{N'}}{\sqrt{N}}.
\end{align}
Then, we get with probability $1-O(N^{-\rho}\log{N})$
\begin{align}
|\gamma_{x}|\le \frac{L^{D}N^{\rho}\sqrt{N'}}{q} \le \frac{L^{D}N^{\rho+\theta/2}}{q}.
\end{align}
According to Lemma \ref{lem: eig val location}, the following inequality holds with overwhelming probability.
\begin{align}\label{eq: eignvalue large gap}
\lambda_{2}-c(N')^{2/3}N^{-2/3}\ge \lambda_{N'}
\end{align}
With overwhelming probability, it follows that 
\begin{align}
\lambda_{2}\gamma_{y} &= \lambda_{2}\langle \yy,\uu_{2}^{(ij)}\rangle \nn\\
&= \langle \yy,\lambda_{2}\uu_{2}^{(ij)}\rangle \nn\\
& = \gamma_{y}\langle \yy, A\yy\rangle + (\lambda_{2}-\mu^{(ij)}_{2})\gamma_{y}+\langle \yy, (B-A)\uu_{2}^{(ij)}\rangle \nn\\
& \le (\lambda_{2}-c(N')^{2/3}N^{-2/3})\gamma_{y}+\frac{L^{D}}{qN}+\frac{L^{D}}{q\sqrt{N}}.
\end{align}
The above inequality implies that
\begin{align}
|\gamma_{y}|\le \frac{L^{D}(N')^{-2/3}N^{1/6}}{q}\le \frac{L^{D}N^{1/6-2\theta/3}}{q}.
\end{align}
Recall $|\alpha|\le\frac{CL^{D}}{q^{2}N}$ (see \eqref{eq: size of alpha} in the proof of Lemma \ref{lem: small perturbation variation}). Since $|\beta|=\sqrt{1-\alpha^{2}-\gamma_{x}^{2}-\gamma_{y}^{2}}\ge 1-|\alpha|-|\gamma_{x}|-|\gamma_{y}|$ and $\lVert \yy \rVert_{\infty}\le \lVert\yy\rVert_{2}\le 1$, by setting $s:=\beta/|\beta|$, we obtain with probability $1-O(N^{-\rho}\log{N})$
\begin{align}
\lVert s\vv_{2}-\uu_{2}^{(ij)}\rVert_{\infty}&\le |\alpha|\lVert\vv_{1}\rVert_{\infty}+(1-|\beta|)\lVert\vv_{2}\rVert_{\infty}+|\gamma_{x}|\lVert \xx\rVert_{\infty}+|\gamma_{y}| \nn\\
&\le |\alpha|\lVert\vv_{1}\rVert_{\infty}+(|\alpha|+|\gamma_{x}|+|\gamma_{y}|)\lVert\vv_{2}\rVert_{\infty}+ |\gamma_{x}|\lVert \xx\rVert_{\infty}+|\gamma_{y}| \nn\\
&\le \frac{L^{D}}{q^{2}N^{3/2}} + \left(\frac{L^{D}}{q^{2}N}+\frac{L^{D}N^{\rho+\theta/2}}{q}+\frac{L^{D}N^{1/6-2\theta/3}}{q}\right)\cdot\frac{L^{D}}{\sqrt{N}}
\nn\\
&\quad + \frac{L^{D}N^{\rho+\theta/2}}{q} \cdot \frac{L^{D}\sqrt{N'}}{\sqrt{N}} + \frac{L^{D}N^{1/6-2\theta/3}}{q}.
\end{align}
To get the desired result, it is enough to show that for some $\delta>0$
\begin{align}
\max\left(\frac{L^{D}N^{\rho+\theta-1/2}}{q}, \frac{L^{D}N^{1/6-2\theta/3}}{q}\right) \le N^{-1/2-\delta}.
\end{align}
with $q^{2}N^{\rho}\gg N$. Recall $q=N^{b}$. We should find appropriate ranges of $q$, $\rho$ and $\theta$ satisfying
%\begin{align}\label{eq: condition for q}
%\begin{cases}
%\rho+\theta-\frac{1}{2}-b<-\frac{1}{2} \\
%\frac{1}{6}-\frac{2}{3}\theta-b<-\frac{1}{2}\\
%2b+\rho > 1
%\end{cases}
%\end{align}
\begin{align}\label{eq: condition for q}
\begin{cases}
\rho+\theta<b, \\
\frac{1}{6}-\frac{2}{3}\theta-b<-\frac{1}{2}, \\
2b+\rho>1.
\end{cases}
\end{align}
If $b>\frac{4}{9}$, % Use Wolfram alpha
then we can find $\rho>0$ and $0<\theta<1$ satisfying the above condition \eqref{eq: condition for q}. 
%Therefore, we finish the proof.

\subsection{Proof of Lemma \ref{lem: expectation estimate}}
We split the expectation into two parts.
\begin{align}
\E\left[ Z_{st}Z_{st}^{[k]}v_{s}v_{t}v_{s}^{[k]}v_{t}^{[k]}\indic_{\mathcal{E}} \right]=
\E\left[ Z_{st}Z_{st}^{[k]}v_{s}v_{t}v_{s}^{[k]}v_{t}^{[k]} \right]
-\E\left[ Z_{st}Z_{st}^{[k]}v_{s}v_{t}v_{s}^{[k]}v_{t}^{[k]}\indic_{\mathcal{E}^{c}} \right]
\end{align}
We consider the upper bound of the second part.
\begin{align}
\E\left[ Z_{st}Z_{st}^{[k]}v_{s}v_{t}v_{s}^{[k]}v_{t}^{[k]}\indic_{\mathcal{E}^{c}} \right]
=\E\left[ Z_{st}Z_{st}^{[k]}v_{s}v_{t}v_{s}^{[k]}v_{t}^{[k]}\indic_{\mathcal{E}_{1}\cap\mathcal{E}_{2}^{c}} \right]+\E\left[ Z_{st}Z_{st}^{[k]}v_{s}v_{t}v_{s}^{[k]}v_{t}^{[k]}\indic_{\mathcal{E}_{1}^{c}} \right]
\end{align}
On the event $\mathcal{E}_{1}$, we bound $v_{s}v_{t}v_{s}^{[k]}v_{t}^{[k]}$ using the delocalization. Since $|Z_{st}Z_{st}^{[k]}|=O(q^{-2})$, and $\prob(\mathcal{E}_{2}^{c})=O(N^{-\rho}\log{N})$ with $q^{2}N^{\rho}\gg N$ by Lemma \ref{lem: eigen vec small perturbation}, we observe that
\begin{align}\label{eq: kapppa estimate 1}
\E\left[ Z_{st}Z_{st}^{[k]}v_{s}v_{t}v_{s}^{[k]}v_{t}^{[k]}\indic_{\mathcal{E}_{1}\cap\mathcal{E}_{2}^{c}} \right]\le
\E\left[\big|Z_{st}Z_{st}^{[k]}\big|\indic_{\mathcal{E}_{1}\cap\mathcal{E}_{2}^{c}}\right]\cdot\frac{L^{D}}{N^{2}}\le \frac{L^{D}}{q^{2}N^{2+\rho}}=o(N^{-3}).
\end{align}
%Since $q^{2}N^{\rho}\gg L^{D}N$ for any $D>0$ by Lemma \ref{lem: eigen vec small perturbation}, we observe that
%\begin{align}\label{eq: kapppa estimate 2}
%\frac{L^{D}}{q^{2}N^{2+\rho}}=o(N^{-3}). % $1-$\kappa<\frac{3}{4}$ means $-\kappa-\frac{3}{4}<-1$
%\end{align}
%Note that the event $\mathcal{E}_{1}^{c}$ holds with overwhelming probability. For all $K>0$, we have
%\begin{align}
%\prob(\mathcal{E}_{1}^{c})=\O(N^{-K}).
%\end{align}
Since $\vv_{2}$ and $\vv_{2}^{[k]}$ are unit vectors,
\begin{align}
\E\left[ Z_{st}Z_{st}^{[k]}v_{s}v_{t}v_{s}^{[k]}v_{t}^{[k]}\indic_{\mathcal{E}_{1}^{c}} \right] = o(N^{-3}).
\end{align}
It turns out that
\begin{align}\label{eq: remove the remainder 1}
\E\left[ Z_{st}Z_{st}^{[k]}v_{s}v_{t}v_{s}^{[k]}v_{t}^{[k]}\indic_{\mathcal{E}} \right]=\E\left[ Z_{st}Z_{st}^{[k]}v_{s}v_{t}v_{s}^{[k]}v_{t}^{[k]} \right]+o(N^{-3}).
\end{align}
To get \eqref{eq: expectation estimate 1}, it is enough to show the following lemma. From now on, we assume $k=N^{\tau}$ with $\tau<2$, so that $kL^{D}\ll N^{2}$ for any $D>0$. We emphasize that this additional assumption on $k$ does not harm the generality due to \eqref{eq: monotonicity} in Lemma \ref{lem: superconcentration}. % consider k, k' such that k'\ll N^{-2} and k'<k
\begin{lem}\label{lem: expectation computation}
	Suppose $k=N^{\tau}$ with $\tau<2$.
	\begin{align}
	\E\left[ Z_{st}Z_{st}^{[k]}v_{s}v_{t}v_{s}^{[k]}v_{t}^{[k]} \right]=\frac{8}{N^{2}(N+1)}\E\left[ \sum_{1\le i,j\le N}v_{i}v_{j}v_{i}^{[k]}v_{j}^{[k]} \right]+o(N^{-3})
	\end{align}
\end{lem}
We postpone the proof of Lemma \ref{lem: expectation computation} to the appendix. What remains is to show \eqref{eq: expectation estimate 2}.
\begin{align}
\E\left[ (\lambda_{2}-\mu_{2})(\lambda_{2}^{[k]}-\mu_{2}^{[k]})\indic_{\mathcal{E}^{c}} \right]
= \E\left[ (\lambda_{2}-\mu_{2})(\lambda_{2}^{[k]}-\mu_{2}^{[k]}) \indic_{\mathcal{E}_{1}\cap\mathcal{E}_{2}^{c}} \right] +  \E\left[ (\lambda_{2}-\mu_{2})(\lambda_{2}^{[k]}-\mu_{2}^{[k]})\indic_{\mathcal{E}_{1}^{c}} \right]
\end{align}
%We can find the desired result applying the similar argument used to prove \eqref{eq: remove the remainder 1}.
To control the first term, we recall \eqref{eq: event 2 effect} and \eqref{eq: kapppa estimate 1}. It follows that
\begin{align}
\E\left[ \big|(\lambda_{2}-\mu_{2})(\lambda_{2}^{[k]}-\mu_{2}^{[k]})\big| \indic_{\mathcal{E}_{1}\cap\mathcal{E}_{2}^{c}} \right]
\le \E\left[\big|Z_{st}Z_{st}^{[k]}\big|\indic_{\mathcal{E}_{1}\cap\mathcal{E}_{2}^{c}}\right]\cdot\frac{L^{D}}{N^{2}}+O\left(\frac{L^{D}}{q^{4}N^{3}}\right) = o(N^{-3}).
\end{align}
For the second term, it is enough to consider the Hilbert-Schmidt norm,
\begin{align}
|(\lambda_{2}-\mu_{2})(\lambda_{2}^{[k]}-\mu_{2}^{[k]})|\le (\lVert A \rVert_{HS}+\lVert B \rVert_{HS})(\lVert A^{[k]} \rVert_{HS}+\lVert B^{[k]} \rVert_{HS}),
\end{align}
and the moment condition \eqref{eq: moment condition} because $\prob(\mathcal{E}_{1}^{c})$ decays very fast.

By the Cauchy–Schwarz, it follows that
\begin{align}
\E\left[ (\lambda_{2}-\mu_{2})(\lambda_{2}^{[k]}-\mu_{2}^{[k]})\indic_{\mathcal{E}_{1}^{c}} \right]=o(N^{-3}),
\end{align}
and we obtain \eqref{eq: expectation estimate 2}.

\section{Small resampling}\label{sec: thm2}

\subsection{Proof of Lemma \ref{lem: main lem for thm2}}

Lemma \ref{lem: main lem for thm2} is a consequence of the following two lemmas.

\begin{lem}\label{lem: lem13 in BLZ}
	Assume $q=N^{b}$ with $b>\frac{1}{3}$ and $k=N^{\tau}$ with $\tau < 2b+\frac{2}{3}$. For $C_{0}>0$ and $C_{1}>0$, the following event holds with overwhelming probability: for all $z=E+i\eta$ such that $|E-2|\le N^{-2/3}L^{C_{0}}$ and $\eta=N^{-2/3}L^{-C_{1}}$,
	\begin{align}
	\max_{1\le i,j\le N} N\eta|R^{[k]}(z)_{ij}-R(z)_{ij}|\le \frac{1}{L^{2}}
	\end{align}
\end{lem}

\begin{lem}\label{lem: lem14 in BLZ}
	Let $q=N^{b}$ with $b>\frac{1}{3}$ and $\tau>0$. Assume $\tau < 2b+\frac{2}{3}$. For $\eps>0$, there exist $C_{1}>0$ such that the following event holds for all $N$ large enough with probability at least $1-\eps$: for all $k\le N^{\tau}$, we have, with $z=\lambda_{2}+i\eta$ and $\eta=N^{-2/3}L^{-C_{1}}$,
	\begin{align}\label{eq: difference between resolvent and vector components}
	\max_{1\le i,j \le N}N|\eta\text{Im}R(z)_{ij}-v_{2,i}v_{2,j}|\le L^{-2}\quad\text{and}\quad
	\max_{1\le i,j \le N}N|\eta\text{Im}R^{[k]}(z)_{ij}-v_{2,i}^{[k]}v_{2,j}^{[k]}|\le L^{-2}.
	\end{align}
\end{lem}

First of all, Choose $C_{1}$ and $C_{2}$ as in Lemma \ref{lem: lem13 in BLZ} and \ref{lem: lem14 in BLZ}. Note that if $q$ and $k$ satisfies the assumption of Lemma \ref{lem: lem13 in BLZ}, we have $k\le N^{5/3}L^{-D}$ for all $D>0$. According to Lemma \ref{lem: eig val location}, there exist $C_{0}$ such that $|\lambda_{2}-2|\le N^{-2/3}L^{C_{0}}$ for large $N$ with overwhelming probability. Thus, we can apply Lemma \ref{lem: lem13 in BLZ} with $z=\lambda_{2}+i\eta$ and $\eta=N^{-2/3}L^{-C_{1}}$. Since
\begin{align}
\lvert v_{2,i}v_{2,j} - v_{2,i}^{[k]}v_{2,j}^{[k]} \rvert 
&\le |v_{2,i}v_{2,j}-\eta\text{Im}R(z)_{ij}| + \eta|\text{Im}R(z)_{ij}-\text{Im}R^{[k]}(z)_{ij}| \nn\\
&\quad + |\eta\text{Im}R^{[k]}(z)_{ij}-v_{2,i}^{[k]}v_{2,j}^{[k]}|,
\end{align}
the desired result follows from Lemma \ref{lem: lem13 in BLZ} and \ref{lem: lem14 in BLZ}. The proof of these two lemmas is in the appendix.\\
\indent In order to show Lemma \ref{lem: lem14 in BLZ}, we need to estimate the effect of the resampling to $\lambda_{2}$. The following proposition provide us the upper bound of the difference between $\lambda_{2}$ and $\lambda_{2}^{[k]}$.

\begin{prop}\label{lem: lem12 in BLZ}
	Let $q=N^{b}$ with $b>\frac{1}{3}$ and $\tau>0$. Assume $\tau < 2b+\frac{2}{3}$. For $C_{3}>0$, the following holds with overwhelming probability: for all $N$ large enough,
	\begin{align}
	\max_{k\le N^{\tau}} |\lambda_{2}-\lambda_{2}^{[k]}|\le N^{-2/3}L^{-C_{3}}.
	\end{align}
	\begin{proof}
		If $\lambda_{2}=\lambda_{2}^{[k]}$, we are done. Thus, suppose $\lambda_{2}^{[k]}<\lambda_{2}$. Choose $C_{0}>0$ as in Lemma \ref{lem: lem9 in BLZ} and set 
		\begin{align}
		\eta:=N^{-2/3}L^{-\frac{C_{0}}{2}-C_{3}}.
		\end{align}
		According to Lemma \ref{lem: lem9 in BLZ}, we can find $1\le i\le N$ such that
		\begin{align}
		\frac{1}{2\eta^{2}} \le N\eta^{-1}\text{Im}R(\lambda_{2}+i\eta)_{ii}
		\end{align} % choose k=2, E=\lambda_{2}
		By Lemma \ref{lem: lem9 in BLZ}, with over whelming probability, we have $|\lambda_{2}-2|\le L^{C_{0}}N^{-2/3}$ and
		\begin{align}
		N\eta^{-1}\text{Im}R^{[k]}(\lambda_{2}+i\eta)_{ii}\le L^{C_{0}}\left(\min_{1\le j\le N}\left|\lambda_{2}-\lambda_{j}^{[k]}\right|\right)^{-2}.
		\end{align}
		Note that $\lambda_{1}^{[k]}$ is close to $f+\frac{1}{f}$ with overwhelming probability, which implies
		\begin{align}
		\min_{1\le j\le N}\left|\lambda_{2}-\lambda_{j}^{[k]}\right| = \left|\lambda_{2}-\lambda_{2}^{[k]}\right|.
		\end{align}
		Since $k\le N^{\tau}$ with $\tau < 2b+\frac{2}{3}$, we can apply Lemma \ref{lem: lem13 in BLZ}. Then, with overwhelming probability,
		\begin{align}
		N\eta^{-1}\text{Im}R^{[k]}(\lambda_{2}+i\eta)_{ii} &\ge N\eta^{-1}\left(\text{Im}R(\lambda_{2}+i\eta)_{ii} 
		- \left|\text{Im}R^{[k]}(\lambda_{2}+i\eta)_{ii} -  \text{Im}R(\lambda_{2}+i\eta)_{ii}\right|\right) \\
		&\ge \frac{1}{2\eta^{2}} - \frac{1}{L^{2}\eta^{2}} \ge \frac{1}{4\eta^{2}}.
		\end{align}
		As a result, we obtain with overwhelming probability
		\begin{align}
		\frac{1}{4\eta^{2}} \le L^{C_{0}}\left|\lambda_{2}-\lambda_{2}^{[k]}\right|^{-2}.
		\end{align}
		In other words, with overwhelming probability,
		\begin{align}
		\left|\lambda_{2}-\lambda_{2}^{[k]}\right| \le 2N^{-2/3}L^{-C_{3}}.
		\end{align}
	\end{proof}
\end{prop}

\section*{Acknowledgment}
This work was supported by by the research grants NRF-2017R1A2B2001952 and NRF-2019R1A5A 1028324. We would also like to show our gratitude to Paul Jung, an associate professor at KAIST, for providing insight and expertise.

\begin{appendices}
\section{Proof of some lemmas}
\subsection{Proof of Lemma \ref{lem: superconcentration}}
According to \cite{Chatterjee05}, we have
\begin{align}
\var\big(f(x)\big) = \frac{1}{2}\sum_{i=1}^{n} \E\left[ \left( f(\mathcal{X}) - f(\mathcal{X}^{(i)}) \right) \left( f(\mathcal{X}^{[i-1]}) - f(\mathcal{X}^{[i]}) \right) \right].
\end{align}
Moreover, for a random permutation $\sigma$ which is uniformly distributed on $S_{n}$, it follows that
\begin{align}
\var\big(f(x)\big) = \frac{1}{2} \sum_{i=1}^{n} \E\left[ \left( f(\mathcal{X}) - f(\mathcal{X}^{\sigma(i)}) \right) \left( f(\mathcal{X}^{\sigma([i-1])}) - f(\mathcal{X}^{\sigma([i])}) \right) \right].
\end{align}
In \cite{BLZ19}, it is also shown that
\begin{multline}
\E\left[\left(f(\mathcal{X})-f(\mathcal{X}^{\sigma(i)})\right) \left(f(\mathcal{X}^{\sigma([i-1])})-f(\mathcal{X}^{\sigma([i])})\right) \right] \\
\ge  \E\left[\left(f(\mathcal{X})-f(\mathcal{X}^{\sigma(i+1)})\right) \left(f(\mathcal{X}^{\sigma([i])})-f(\mathcal{X}^{\sigma([i+1])})\right) \right] 
\end{multline}
for $i\in\{1,\cdots, n-1\}$ and
\begin{align}
\E\left[\left(f(\mathcal{X})-f(\mathcal{X}^{\sigma(n)})\right) \left(f(\mathcal{X}^{\sigma([n-1])})-f(\mathcal{X}^{\sigma([n])})\right) \right] \ge 0.
\end{align}
Consequently, we have for each $k\in\{1,\cdots, n-1\}$
\begin{align}
2\var\big(f(x)\big) & \ge \sum_{i=1}^{k} \E\left[ \left( f(\mathcal{X}) - f(\mathcal{X}^{\sigma(i)}) \right) \left( f(\mathcal{X}^{\sigma([i-1])}) - f(\mathcal{X}^{\sigma([i])}) \right) \right] \nn \\
& \ge k \; \E\left[ \left( f(\mathcal{X}) - f(\mathcal{X}^{\sigma(k)}) \right) \left( f(\mathcal{X}^{\sigma([k-1])}) - f(\mathcal{X}^{\sigma([k])}) \right) \right],
\end{align}
which implies
\begin{align}\label{lem: BLZ lem2}
\E\left[ \left( f(\mathcal{X}) - f(\mathcal{X}^{\sigma(k)}) \right) \left( f(\mathcal{X}^{\sigma([k-1])}) - f(\mathcal{X}^{\sigma([k])}) \right) \right] \le \frac{2\var\big(f(x)\big)}{k}.
\end{align}
We split $I_{i}$ into two parts.
\begin{align}
I_{i} &= \E\left[ \left(f(\mathcal{X})-f(\mathcal{X}^{(j)})\right) \left(f(\mathcal{X}^{\sigma([i-1])})-f(\mathcal{X}^{(j)\circ\sigma([i-1])})\right) \right] \nn\\
& = \E\left[ \left(f(\mathcal{X})-f(\mathcal{X}^{(j)})\right) \left(f(\mathcal{X}^{\sigma([i-1])})-f(\mathcal{X}^{(j)\circ\sigma([i-1])})\right)\; ; \; j\in\sigma([i-1]) \right] \nn \\
& \quad + \E\left[ \left(f(\mathcal{X})-f(\mathcal{X}^{(j)})\right) \left(f(\mathcal{X}^{\sigma([i-1])})-f(\mathcal{X}^{(j)\circ\sigma([i-1])})\right)\; ; \; j\notin\sigma([i-1]) \right]
\end{align}
%where the outer expectation is with respect to $\sigma$,.and the inner expectation is conditioned on $j$ and $\sigma$.
By recalling \eqref{eq: decomposition of expec}, we have
\begin{multline}
\E\left[ \left(f(\mathcal{X})-f(\mathcal{X}^{(j)})\right) \left(f(\mathcal{X}^{\sigma([i-1])})-f(\mathcal{X}^{(j)\circ\sigma([i-1])})\right)\; ; \; j\notin\sigma([i-1]) \right] \nn \\
\quad = \E\left[ \left( f(\mathcal{X}) - f(\mathcal{X}^{\sigma(i)}) \right) \left( f(\mathcal{X}^{\sigma([i-1])}) - f(\mathcal{X}^{\sigma([i])}) \right) \right]  \left(\frac{n-(i-1)}{n}\right).
\end{multline}
% The last equality follows because the transposition $(\sigma(i), \sigma(i+...))$ is bijection on $S_{n}$. 
Also, it follows that
\begin{align}
\E&\left[ \left(f(\mathcal{X})-f(\mathcal{X}^{(j)})\right) \left(f(\mathcal{X}^{\sigma([i-1])})-f(\mathcal{X}^{(j)\circ\sigma([i-1])})\right)\; ; \; j\in\sigma([i-1]) \right] \nn \\
&\;\; = \E\left[ \left(f(\mathcal{X})-f(\mathcal{X}^{(j)})\right) \left( f(\mathcal{X}^{\sigma([i-1])\backslash\{j\}})-f(\mathcal{X}^{(j)\circ\sigma([i-1])})\right)\; ; \; j\in\sigma([i-1]) \right] \nn\\
& \;\;\quad - \E\left[ \left(f(\mathcal{X})-f(\mathcal{X}^{(j)})\right) \left(f(\mathcal{X}^{\sigma([i-1])\backslash\{j\}})- f(\mathcal{X}^{\sigma([i-1])})\right)\; ; \; j\in\sigma([i-1]) \right] \nn\\
& \;\; = \E\left[ \left( f(\mathcal{X}) - f(\mathcal{X}^{\sigma(i-1)}) \right) \left( f(\mathcal{X}^{\sigma([i-2])}) - f(\mathcal{X}^{\sigma([i-1])}) \right) \right]  \left(\frac{i-1}{n}\right) \nn \\
&\;\;\quad - \E\left[ \left(f(\mathcal{X})-f(\mathcal{X}^{(j)})\right) \left(f(\mathcal{X}^{\sigma([i-1])\backslash\{j\}})- f(\mathcal{X}^{\sigma([i-1])})\right)\; ; \; j\in\sigma([i-1]) \right].
\end{align}
% The last equality follows because the transposition $(\sigma(i-1-...), \sigma(i-1))$ is bijection on $S_{n}$. 
To finish the proof, we claim that for a fixed $\sigma\in S_{n}$ and any $j\in \sigma([i-1])$,
\begin{align}\label{eq: claim for monoticity}
\E\left[ \big( f(\mathcal{X})-f(\mathcal{X}^{(j)}) \big) \big( f(\mathcal{X}^{\sigma([i-1])\backslash\{j\}})-f(\mathcal{X}^{\sigma([i-1])}) \big)\right] \ge 0, \quad i\in[n+1].
\end{align}
If the claim holds, we can obtain
\begin{align}
I_{i}&\le \E\left[ \left( f(\mathcal{X}) - f(\mathcal{X}^{\sigma(i)}) \right) \left( f(\mathcal{X}^{\sigma([i-1])}) - f(\mathcal{X}^{\sigma([i])}) \right) \right] \left(\frac{n-(i-1)}{n}\right) \nn\\
&\quad + \E\left[ \left( f(\mathcal{X}) - f(\mathcal{X}^{\sigma(i-1)}) \right) \left( f(\mathcal{X}^{\sigma([i-2])}) - f(\mathcal{X}^{\sigma([i-1])}) \right) \right] \left(\frac{i-1}{n}\right).
\end{align}
By \eqref{lem: BLZ lem2}, it follows that
\begin{align}
I_{i}\le \frac{2\text{Var}(f(\mathcal{X}))}{n}\left(1+\frac{n-i+1}{i} \right) = \frac{2\text{Var}(f(\mathcal{X}))}{i} \left(\frac{n+1}{n}\right).
\end{align}
What remains is to show the claim \eqref{eq: claim for monoticity}. For simplicity, we assume $\sigma$ is the identity, i.e., $\sigma(k)=k$ for all $k\in[n]$, and $j=1$. In this case,
\begin{align}
I_{i}':=\E\left[ \big( f(\mathcal{X})-f(\mathcal{X}^{(1)}) \big) \big( f(\mathcal{X}^{[i-1]\backslash\{1\}})-f(\mathcal{X}^{[i-1]}) \big)\right]
\end{align}
Note that $\mathcal{X}^{(1)}=(\mathcal{X}_{1}'',\mathcal{X}_{2},\cdots,\mathcal{X}_{n})$. It is enough to show that $I_{i}'\ge I_{i+1}'$ and $I_{n+1}'\ge 0$. We shall use the following notation:
\begin{align}
J&:=(\mathcal{X}_{2},\cdots,\mathcal{X}_{i-1}), \nn\\
J'&:=(\mathcal{X}_{2}',\cdots,\mathcal{X}_{i-1}'), \nn\\
K&:=(\mathcal{X}_{i+1},\cdots,\mathcal{X}_{n}).
\end{align}
We have
\begin{align}
I_{i}'&=\E\left[ \big( f(\mathcal{X}_{1},J,\mathcal{X}_{i},K) - f(\mathcal{X}_{1}'',J,\mathcal{X}_{i},K) \big) \big( f(\mathcal{X}_{1},J',\mathcal{X}_{i},K) - f(\mathcal{X}_{1}',J',\mathcal{X}_{i},K) \big) \right], \\
I_{i+1}'&=\E\left[ \big( f(\mathcal{X}_{1},J,\mathcal{X}_{i},K) - f(\mathcal{X}_{1}'',J,\mathcal{X}_{i},K) \big) \big( f(\mathcal{X}_{1},J',\mathcal{X}_{i}',K) - f(\mathcal{X}_{1}',J',\mathcal{X}_{i}',K) \big) \right].
\end{align}
Next, we get
\begin{align}
\E&\left[\big( f(\mathcal{X}_{1},J,\mathcal{X}_{i},K) - f(\mathcal{X}_{1}'',J,\mathcal{X}_{i},K) \big) \big( f(\mathcal{X}_{1},J',\mathcal{X}_{i},K) - f(\mathcal{X}_{1}',J',\mathcal{X}_{i},K) \big)
\middle| \mathcal{X}_{1},\mathcal{X}_{i},K \right] \\
%&=\E\left[f(\mathcal{X}_{1},J,\mathcal{X}_{i},K) - f(\mathcal{X}_{1}'',J,\mathcal{X}_{i},K) \middle| \mathcal{X}_{1},\mathcal{X}_{i},K \right]
%\E\left[f(\mathcal{X}_{1},J',\mathcal{X}_{i},K) - f(\mathcal{X}_{1}',J',\mathcal{X}_{i},K) \middle| \mathcal{X}_{1},\mathcal{X}_{i},K \right] \\
&=\E\left[f(\mathcal{X}_{1},J,\mathcal{X}_{i},K) - f(\mathcal{X}_{1}'',J,\mathcal{X}_{i},K) \middle| \mathcal{X}_{1},\mathcal{X}_{i},K \right]^{2},
\end{align}
where we use $(\mathcal{X}_{1}'',J)$ and $(\mathcal{X}_{1}',J')$ i.i.d. in the last equality.
Similarly,
\begin{align}
\E&\left[\big( f(\mathcal{X}_{1},J,\mathcal{X}_{i},K) - f(\mathcal{X}_{1}'',J,\mathcal{X}_{i},K) \big) \big( f(\mathcal{X}_{1},J',\mathcal{X}_{i}',K) - f(\mathcal{X}_{1}',J',\mathcal{X}_{i}',K) \big)
\middle| \mathcal{X}_{1},\mathcal{X}_{i},\mathcal{X}_{i}',K \right] \nn\\
&=\E\left[f(\mathcal{X}_{1},J,\mathcal{X}_{i},K) - f(\mathcal{X}_{1}'',J,\mathcal{X}_{i},K) \middle| \mathcal{X}_{1},\mathcal{X}_{i},K \right]
\nn \\
&\quad\times \E\left[f(\mathcal{X}_{1},J',\mathcal{X}_{i}',K) - f(\mathcal{X}_{1}',J',\mathcal{X}_{i}',K) \middle| \mathcal{X}_{1},\mathcal{X}_{i}',K \right] \nn \\
&=\E\left[f(\mathcal{X}_{1},J,\mathcal{X}_{i},K) - f(\mathcal{X}_{1}'',J,\mathcal{X}_{i},K) \middle| \mathcal{X}_{1},\mathcal{X}_{i},K \right] \nn\\
&\quad\times \E\left[f(\mathcal{X}_{1},J,\mathcal{X}_{i}',K) - f(\mathcal{X}_{1}'',J,\mathcal{X}_{i}',K) \middle| \mathcal{X}_{1},\mathcal{X}_{i}',K \right].
\end{align}
Let us define $h(\cdot,\cdot,\cdot)$ by setting for any independent random variables $x,y,z$
\begin{align}
h(x,y,z)=\E\left[f(x,J,y,z) - f(\mathcal{X}_{1}'',J,y,z) \middle| x,y,z\right].
\end{align}
Then, we observe $I_{i+1}'\le I_{i}'$ because
\begin{align}
I_{i+1}'&=\E\left[ h(\mathcal{X}_{1},\mathcal{X}_{i},K)h(\mathcal{X}_{1},\mathcal{X}_{i}',K) \right] \nn\\
&=\E\bigg[ \E [ h(\mathcal{X}_{1},\mathcal{X}_{i},K)h(\mathcal{X}_{1},\mathcal{X}_{i}',K) | \mathcal{X}_{1},K ]\bigg] \nn\\
%&=\E\bigg[ \E [ h(\mathcal{X}_{1},\mathcal{X}_{i},K) | \mathcal{X}_{1},K ] \E [ h(\mathcal{X}_{1},\mathcal{X}_{i}',K) | \mathcal{X}_{1},K ]\bigg] \nn\\
&= \E\left[ \E [ h(\mathcal{X}_{1},\mathcal{X}_{i},K) | \mathcal{X}_{1},K ] ^{2} \right] \nn\\
&\le \E\left[ \E [ h(\mathcal{X}_{1},\mathcal{X}_{i},K)^{2} | \mathcal{X}_{1},K ]  \right] = \E\left[  h(\mathcal{X}_{1},\mathcal{X}_{i},K)^{2}  \right] = I_{i}'.
\end{align}
Next, we will deal with $I_{n+1}'$.
\begin{align}
I_{n+1}' = \E\left[ \big( f(\mathcal{X})-f(\mathcal{X}^{(1)}) \big) \big( f(\mathcal{X}^{[n]\backslash\{1\}})-f(\mathcal{X}^{[n]}) \big)\right]
\end{align}
We observe
\begin{align}
\E\left[ f(\mathcal{X})f(\mathcal{X}^{[n]\backslash\{1\}}) \right] &= \E\left[ \E[ f(\mathcal{X})f(\mathcal{X}^{[n]\backslash\{1\}}) | \mathcal{X}_{1}]\right] = \E\left[ \E[ f(\mathcal{X}) | \mathcal{X}_{1}]^{2}\right].
\end{align}
Let us denote the independence between random variables by the notation $\indep$. Since 
\begin{align}
\mathcal{X}^{(1)}\indep \mathcal{X}^{[n]\backslash\{1\}},\quad \mathcal{X} \indep \mathcal{X}^{[n]}\quad \text{and}\quad \mathcal{X}^{(1)}\indep \mathcal{X}^{[n]},
\end{align}
we have
\begin{align}
\E\left[ f(\mathcal{X}^{(1)})f(\mathcal{X}^{[n]\backslash\{1\}}) \right]= \E\left[ f(\mathcal{X})f(\mathcal{X}^{[n]}) \right] = \E\left[ f(\mathcal{X}^{(1)})f(\mathcal{X}^{[n]})\right]=\E\left[ f(\mathcal{X}) \right]^{2}.
\end{align}
Therefore, by Jensen's inequality, we obtain
\begin{align}
I_{n+1}' = \E\left[\E[ f(\mathcal{X}) | \mathcal{X}_{1}]^{2}\right] - \E\left[ f(\mathcal{X}) \right]^{2} = \E\left[\E[ f(\mathcal{X}) | \mathcal{X}_{1}]^{2}\right] - \E\left[ \E[ f(\mathcal{X}) | \mathcal{X}_{1}] \right]^{2} \ge 0.
\end{align}

\subsection{Proof of Lemma \ref{lem: expectation computation}}\label{sec: expectation computation}
Let $\E_{st}$ be the conditional expectation given $(A,A',A'',S_{k})$. Under $\E_{st}$, the random pair $(st)$ is integrated out.
\begin{align}\label{eq: int out random pair}
\E_{st}\left[ Z_{st}Z_{st}^{[k]}v_{s}v_{t}v_{s}^{[k]}v_{t}^{[k]} \right] = \frac{2}{N(N+1)}\sum_{1\le i\le j\le N}Z_{ij}Z_{ij}^{[k]}v_{i}v_{j}v_{i}^{[k]}v_{j}^{[k]}
\end{align}
Set $Z'_{ij}:=(a_{ij}'-a_{ij}'')(1+\indic(i\neq j))$. Then, the sum in the right-hand side of \eqref{eq: int out random pair} is split up into two parts,
\begin{align}
\sum_{(ij)\in S_{k}}Z_{ij}Z'_{ij}v_{i}v_{j}v_{i}^{[k]}v_{j}^{[k]} +
\sum_{(ij)\notin S_{k}}Z_{ij}^{2}v_{i}v_{j}v_{i}^{[k]}v_{j}^{[k]},
\end{align}
where $1\le i\le j\le N$ in the both sums. Note that
\begin{align}
\E\left[ Z_{ij}Z'_{ij} \right] = \begin{cases}
\frac{4}{N} &\text{if } i\neq j,\\
\frac{1}{N} &\text{if } i=j,
\end{cases}
\quad\text{and}\quad
\E\left[ Z_{ij}^{2} \right] = \begin{cases}
\frac{8}{N} &\text{if } i\neq j,\\
\frac{2}{N} &\text{if } i=j.
\end{cases}
\end{align}
Next, we shall calculate
\begin{align}
\sum_{1\le i\le j\le N}\E\left[Z_{ij}Z_{ij}^{[k]}\right]v_{i}v_{j}v_{i}^{[k]}v_{j}^{[k]}.
\end{align}
We deal with two cases, $(ij)\in S_{k}$ and $(ij)\notin S_{k}$, separately. 
\begin{align}
\sum_{(ij)\in S_{k}}\E\left[Z_{ij}Z_{ij}'\right]v_{i}v_{j}v_{i}^{[k]}v_{j}^{[k]} = \frac{4}{N}\sum_{\substack{(ij)\in S_{k}\\i<j}}v_{i}v_{j}v_{i}^{[k]}v_{j}^{[k]}+\frac{1}{N}\sum_{\substack{(ij)\in S_{k}\\i=j}}v_{i}v_{j}v_{i}^{[k]}v_{j}^{[k]}
\end{align}
\begin{align}
\sum_{(ij)\notin S_{k}}\E\left[Z_{ij}^{2}\right]v_{i}v_{j}v_{i}^{[k]}v_{j}^{[k]} = \frac{8}{N}\sum_{\substack{(ij)\notin S_{k}\\i<j}}v_{i}v_{j}v_{i}^{[k]}v_{j}^{[k]}+\frac{2}{N}\sum_{\substack{(ij)\notin S_{k}\\i=j}}v_{i}v_{j}v_{i}^{[k]}v_{j}^{[k]}
\end{align}
For brevity, we denote $v_{i}v_{j}v_{i}^{[k]}v_{j}^{[k]}$ by $V_{i,j,k}$. Combining above two equations, we obtain
\begin{multline}
\sum_{1\le i\le j\le N}\E\left[Z_{ij}Z_{ij}^{[k]}\right]V_{i,j,k}=\\
\frac{8}{N}\sum_{i<j}V_{i,j,k}+\frac{4}{N}\sum_{i=j}V_{i,j,k}
-\frac{4}{N}\sum_{\substack{(ij)\in S_{k}\\i<j}}V_{i,j,k}-\frac{3}{N}\sum_{\substack{(ij)\in S_{k}\\i=j}}V_{i,j,k}-\frac{2}{N}\sum_{\substack{(ij)\notin S_{k}\\i=j}}V_{i,j,k}
\end{multline}
In fact, the last three sums of the above equation is negligible after taking expectation. The delocalization result (Lemma \ref{lem: delocalization}) and the Cauchy–Schwarz is used. Recall the event $\mathcal{E}_{1}$ defined previously. The delocalization occurs on $\mathcal{E}_{1}$ and $\prob(\mathcal{E}_{1}^{c})\le N^{-K}$ for an arbitrary $K>0$. Since we assume $|S_{k}|=k=N^{\tau}$ with $\tau<2$, we have
\begin{align}
\E\left[ \frac{1}{N}\sum_{\substack{(ij)\in S_{k}\\i<j}}V_{i,j,k} \left( \indic_{\mathcal{E}_{1}}+\indic_{\mathcal{E}_{1}^{c}} \right)\right]
\le \frac{1}{N}\cdot k \cdot \frac{L^{D}}{N^{2}} + \frac{1}{N}\cdot k\cdot \frac{1}{N^{K/2}}
= o(N^{-1}).
\end{align}
Using the similar argument, we get
\begin{align}
\E\left[ \frac{1}{N}\sum_{\substack{(ij)\in S_{k}\\i=j}}V_{i,j,k} \right] 
\le \frac{1}{N}\cdot N \cdot \frac{L^{D}}{N^{2}} + \frac{1}{N}\cdot N\cdot \frac{1}{N^{K/2}} = o(N^{-1})
\end{align}
and
\begin{align}
\E\left[ \frac{1}{N}\sum_{\substack{(ij)\notin S_{k}\\i=j}}V_{i,j,k} \right] 
\le \frac{1}{N}\cdot N \cdot \frac{L^{D}}{N^{2}} + \frac{1}{N}\cdot N\cdot \frac{1}{N^{K/2}} = o(N^{-1}).
\end{align}
Since $2\sum_{i<j}V_{i,j,k}=\sum_{i\neq j}V_{i,j,k}$ by symmetry, it follows that
\begin{align}
\E\left[\sum_{1\le i\le j\le N}\E\left[Z_{ij}Z_{ij}^{[k]}\right]V_{i,j,k}\right]=
\frac{4}{N}\E\left[\sum_{1\le i,j \le N}V_{i,j,k} \right] + o(N^{-1}).
\end{align}
What remains is to show
\begin{align}
\E\left[\sum_{1\le i\le j\le N}\left(Z_{ij}Z_{ij}^{[k]}-\E\left[Z_{ij}Z_{ij}^{[k]}\right]\right)V_{i,j,k}\right]=
o(N^{-1}).
\end{align}
Let $A'''=(a_{ij}''')$ be a independent copy of $A$ which is also independent of $A'$ and $A''$. As we set $B_{(ij)}$ and $B_{(ij)}^{[k]}$, analogously define $\tilde{B}_{(ij)}$ and $\tilde{B}_{(ij)}^{[k]}$ by using $a_{ij}'''$ and $a_{ji}'''$ instead of $a_{ij}''$ and $a_{ji}''$. Denote by $\tilde{\uu}_{(ij)}$ and $\tilde{\uu}_{(ij)}^{[k]}$ be the second top eigenvector of $\tilde{B}_{(ij)}$ and $\tilde{B}_{(ij)}^{[k]}$ respectively. To ease the notation, we define
\begin{align}
\tilde{U}_{i,j,k}:=\left(\tilde{\uu}_{(ij)}\right)_{i}\left(\tilde{\uu}_{(ij)}\right)_{j}\left(\tilde{\uu}_{(ij)}^{[k]}\right)_{i}\left(\tilde{\uu}_{(ij)}^{[k]}\right)_{j}
\end{align}
and
\begin{align}
W_{ij}:=Z_{ij}Z_{ij}^{[k]}-\E\left[Z_{ij}Z_{ij}^{[k]}\right].
\end{align}
We can find that
\begin{align}
\E\left[ W_{ij} \tilde{U}_{i,j,k} \right] = \E\left[W_{ij}\right]\cdot \E\left[\tilde{U}_{i,j,k} \right] = 0
\end{align}
because $Z_{ij}$ and $Z_{ij}^{[k]}$ only depend on $X_{ij}$, $X_{ij}'$ and $X_{ij}''$ while $\tilde{u}_{(ij)}$ and $\tilde{u}_{(ij)}^{[k]}$ is independent of $X_{ij}$, $X_{ij}'$ and $X_{ij}''$. Thus, it is enough to show
\begin{align}
\E\left[\sum_{1\le i\le j\le N}W_{ij}\left( V_{i,j,k} - \tilde{U}_{i,j,k} \right)\right]=o(N^{-1}).
\end{align}
%Define the events $\mathcal{E}_{1}'$ and $\mathcal{E}_{2}'$ analogously for $\{\tilde{\uu}_{(ij)}, \tilde{\uu}_{(ij)}^{[k]}\}_{1\le i\le j\le N}$.
Let us define the event $\mathcal{E}_{1}'$: for some $D>0$
\begin{align}
\max_{1\le i\le j\le N}(\lVert \vv_{2}  \rVert_{\infty},\lVert \tilde{\uu}_{(ij)} \rVert_{\infty},\lVert  \vv_{2}^{[k]}\rVert_{\infty},\lVert \tilde{\uu}_{(ij)}^{[k]} \rVert_{\infty}) \le \frac{L^{D}}{\sqrt{N}}.
\end{align}
We also define the event $\mathcal{E}_{2}'$:
\begin{align}
\max_{1\le i,j \le N}\lVert \vv_{2} - \uu_{(ij)} \rVert_{\infty}\le N^{-\frac{1}{2}-\delta} \quad \text{and}
\max_{1\le i,j \le N}\lVert \vv_{2}^{[k]} - \uu_{(ij)}^{[k]} \rVert_{\infty}\le N^{-\frac{1}{2}-\delta}.
\end{align}
On the event $\mathcal{E}':=\mathcal{E}_{1}'\cap \mathcal{E}_{2}'$, we can use the bound
\begin{align}
|V_{i,j,k} - \tilde{U}_{i,j,k}|\le \frac{L^{D}}{N^{2+\alpha}}.
\end{align}
On the event $\mathcal{E}_{1}'\cap(\mathcal{E}_{2}')^{c}$, it follows that from the delocalization
\begin{align}
|V_{i,j,k} - \tilde{U}_{i,j,k}|\le \frac{L^{D}}{N^{2}}.
\end{align}
We shall compute the expectation on $\mathcal{E}'$. Since $\E[Z_{ij}Z_{ij}^{[k]}]=O(N^{-1})$, we observe
\begin{align}
\E\left[\sum_{1\le i\le j\le N}W_{ij}\left( V_{i,j,k} - \tilde{U}_{i,j,k} \right)\indic_{\mathcal{E}'}\right]=O\left(N^{2}\cdot\frac{1}{N}\cdot \frac{L^{D}}{N^{2+\alpha}} \right).
\end{align}% See A4 note on 7/23 for more detail
On the event $\mathcal{E}_{1}'\cap(\mathcal{E}_{2}')^{c}$, we have
\begin{align}
\E\left[\sum_{1\le i\le j\le N}W_{ij}\left( V_{i,j,k} - \tilde{U}_{i,j,k} \right)\indic_{\mathcal{E}_{1}'\cap(\mathcal{E}_{2}')^{c}}\right]
\le N^{2}\cdot\frac{C}{q^{2}}\cdot \frac{L^{D}}{N^{2}}\cdot \frac{1}{N^{\kappa'}}
+N^{2}\cdot\frac{C}{N}\cdot \frac{L^{D}}{N^{2}}\cdot \frac{1}{N^{\kappa'}} 
\end{align}
for some $\kappa'>0$ given by Lemma \ref{lem: eigen vec small perturbation}. We can verify that
\begin{align}
\E\left[\sum_{1\le i\le j\le N}W_{ij}\left( V_{i,j,k} - \tilde{U}_{i,j,k} \right)\indic_{\mathcal{E}_{1}'\cap(\mathcal{E}_{2}')^{c}}\right] = o(N^{-1}).
\end{align}
Since $|V_{i,j,k} - \tilde{U}_{i,j,k}|\le 2$, on the event $(\mathcal{E}_{1}')^{c}$, we get
\begin{align}
\E\left[\sum_{1\le i\le j\le N}W_{ij}\left( V_{i,j,k} - \tilde{U}_{i,j,k} \right)\indic_{(\mathcal{E}_{2}')^{c}}\right]
\le N^{2}\cdot \frac{C}{q^{2}}\cdot\frac{1}{N^{K'}} =o(N^{-1})
\end{align}
by choosing $K'>0$ large enough. Therefore, we have the desired result.	

\subsection{Proof of Lemma \ref{lem: lem13 in BLZ}}

\begin{lem}\label{lem: lem9 in BLZ}
	For any $1\le k\le N$, there exists an integer $1\le i\le N$ such that for all $E$ and $\eta>0$
	\begin{align}
	\frac{1}{2\big(\max(\eta,|\lambda_{k}-E|)\big)^{2}}\le N\eta^{-1}\text{Im}R(E+i\eta)_{ii}.
	\end{align}
	Moreover, let $1\le k\le L$. There exists $C_{0}>0$ such that with overwhelming probability, we have
	\begin{align}
	|\lambda_{k}-2|<L^{C_{0}}N^{-2/3}
	\end{align}
	and for all integers $1\le i\le N$ and all $E$ satisfying $|E-2|<L^{C_{0}}N^{-2/3}$,
	\begin{align}
	N\eta^{-1}\text{Im}R(E+i\eta)_{ii}\le L^{C_{0}}\left(\min_{1\le j\le N}|\lambda_{j}-E|\right)^{-2}
	\end{align}
	\begin{proof}
		% By spectral therem
		\begin{align}
		\text{Im}R_{ii}=\sum_{p=1}\frac{\eta(\vv_{p}(i))^{2}}{(\lambda_{p}-E)^{2}+\eta^{2}}
		\end{align}
		Note that $\vv_{k}(i)\ge N^{-1/2}$ for some $1\le i\le N$. Since
		\begin{align}
		N\eta^{-1}\text{Im}R_{ii}\ge \frac{N(\vv_{k}(i))^{2}}{(\lambda_{k}-E)^{2}+\eta^{2}}\ge \frac{1}{2\big(\max(\eta,|\lambda_{k}-E|)\big)^{2}},
		\end{align}
		the first one follows.\\
		Now we choose an integer $1\le k\le L$. By Lemma \ref{lem: eig val location}, with overwhelming probability, we have for some constant $C_{0}>0$
		\begin{align}
		|\lambda_{k}-2|\le |\lambda_{k}-\gamma_{k}|+|\gamma_{k}-2| \le L^{C_{0}}N^{-2/3},
		\end{align}
		which is the second result.\\
		Next, we consider $E$ satisfying $|E-2|<L^{C_{0}}N^{-2/3}$. By Lemma \ref{lem: eig val location} and Lemma \ref{lem: delocalization}, for some constants $c>0$, the following event $\mathcal{E}$ holds with overwhelming probability: we have $E-\lambda_{p}\ge cp^{2/3}N^{-2/3}$ for all integer $p> N':=\lfloor L^{3C_{0}}\rfloor$, and $\max_{1\le i \le N}\lVert\vv_{i} \rVert_{\infty}^{2}\le L/N$. On the event $\mathcal{E}$, we have for some $C>0$
		\begin{align}
		\sum_{p=N'+1}^{N}\frac{N(\vv_{p}(i))^{2}}{(\lambda_{p}-E)^{2}+\eta^{2}}\le
		\sum_{p=N'+1}^{N}\frac{L}{(\lambda_{p}-E)^{2}}\le CL(N')^{-1/3}N^{4/3}
		\end{align}
		and
		\begin{align}
		\sum_{p=1}^{N'}\frac{N(\vv_{p}(i))^{2}}{(\lambda_{p}-E)^{2}+\eta^{2}} \le
		\frac{LN'}{\left(\min_{1\le j\le N}|\lambda_{j}-E|\right)^{2}}.
		\end{align}
		By Lemma \ref{lem: eig val location}, we have $CL(N')^{-1/3}N^{4/3}\le LN'\left(\min_{1\le j\le N}|\lambda_{j}-E|\right)^{-2}$ with overwhelming probability, which implies the third result by adjusting the value of $C_{0}$.
		% see note on Aug-7 for more detail.
	\end{proof}
	
\end{lem}

\begin{lem}\label{lem: lem10 in BLZ}
	Let $C_{0},C_{1}>0$. There exists $C>0$ such that, with overwhelming probability, the following event holds: for all $z=E+i\eta$ such that $|E-2|\le L^{C_{0}}N^{-2/3}$ and $\eta=L^{-C_{1}}N^{-2/3}$, we have for $i\neq j$
	\begin{align}
	|R(z)_{ij}|\le L^{C}N^{-1/3} \quad\text{and}\quad |R(z)_{ii}|\le C.
	\end{align}
	\begin{proof}
		We shall use the local law \cite[Theorem 2.9]{EKYY13}. Then, for $i\neq j$, the following holds with overwhelming probability:
		\begin{align}
		|R(z)_{ij}|\le \delta
		\quad\text{and}\quad
		|R(z)_{ii}-m_{\text{sc}}(z)|\le\delta
		\end{align}
		where for some constant $C'>0$
		\begin{align}
		\delta:=L^{C'}\left(\frac{1}{q}+\sqrt{\frac{\text{Im}\;m_{\text{sc}}(z)}{N\eta}}+\frac{1}{N\eta}\right)
		\end{align} 
		and $m_{\text{sc}}(z)$ is the Cauchy-Stieltjes transform of the semicircle distribution. According to \cite[Lemma 3.2]{KY13},
		\begin{align}
		|m_{\text{sc}}(z)|\asymp 1
		\quad\text{and}\quad
		\text{Im}\;m_{\text{sc}}(z)\asymp\begin{cases}
		\sqrt{|E-2|+\eta} & \text{if }|E|\le 2, \\
		\frac{\eta}{\sqrt{|E-2|+\eta}} & \text{if } |E|\ge 2.
		\end{cases}
		\end{align}
		Thus, for $|E-2|\le L^{C_{0}}N^{-2/3}$ and $\eta=L^{-C_{1}}N^{-2/3}$, it follows for some $C>0$
		\begin{align}
		\text{Im}\;m_{\text{sc}}(z) \le L^{C}N^{-1/3}.
		\end{align}
		Recall that we consider the regime $q=N^{b}$ with $b>\frac{1}{3}$. Therefore, for some $C>0$, we obtain
		\begin{align}
		\delta \le L^{C}N^{-1/3}.
		\end{align}
		The desired result follows.
		% see note on Aug-7 for more detail.
	\end{proof}	
\end{lem}

\begin{proof}[Proof of Lemma \ref{lem: lem13 in BLZ}]
	For $0\le t\le k$, let $\FF_{t}$ be the $\sigma$-algebra generated by the random variable $A$, $S_{k}$ and $(a'_{i_{s}j_{s}})_{1\le s\le t}$. For $1\le i,j\le N$, we set
	\begin{align}
	T_{ij}=\{ 1\le t\le k :\{i_{t},j_{t}\}\cap\{i,j\}\neq\emptyset \}
	\end{align}
	Since $\FF_{0}$ contains $\sigma(S_{k})$, the cardinality $|T_{ij}|$ is $\FF_{0}$-measurable. We write
	\begin{align}
	|T_{ij}|=&\sum_{l=i}^{N}\indic[(il)\in S_{k}] + \sum_{l=1}^{i-1} \indic[(li)\in S_{k}] \nn\\
	& + \indic(i\neq j) \left( \sum_{l=j}^{N}\indic[(jl)\in S_{k}] + \sum_{l=1}^{j-1} \indic[(lj)\in S_{k}] - 1 \right)
	\end{align}
	Since $S_{k}$ is uniformly chosen at random, we have for any $1\le i'\le j'\le N$
	\begin{align}
	\prob((i'j')\in S_{k})=\frac{\begin{pmatrix}
		\frac{N(N+1)}{2}-1 \\ k-1
		\end{pmatrix}}{\begin{pmatrix}
		\frac{N(N+1)}{2} \\ k
		\end{pmatrix}} = \frac{2k}{N(N+1)}.
	\end{align}
	Thus, it follows that
	\begin{align}
	\E|T_{ij}|=\begin{cases}
	\frac{2k(2N-1)}{N(N+1)} & \text{if } i < j,\\
	\\
	\frac{2k}{(N+1)} & \text{if } i=j.
	\end{cases}
	\end{align}
	In other words,
	\begin{align}
	\E|T_{ij}| = \Theta\left(\frac{k}{N}\right)
	\end{align}
	Next, we apply \cite[Proposition 1.1]{Chatterjee07}.
	%(Set $n=\frac{N(N+1)}{2}$. Let $\pi$ to be a random permutation of $\{1,\cdots, \frac{N(N+1)}{2}\}$. Then, $S_{k}=\{\pi(1),\cdots, \pi(k)\}$. For $1\le i'\le k$, define $a_{i'j'}$ to be $1$ if the pair corresponding to $j'$ has $i$ or $j$. Otherwise, $a_{i'j'}=0$. For $i'>k$, $a_{i'j'}=0$.)
	Then, for any $u>0$, we have
	\begin{align}
	\prob\left( |T_{ij}|\ge \E|T_{ij}| +u \right) \le \exp\left( -\frac{u^{2}}{4\E|T_{ij}|+2u} \right),
	\end{align}
	which implies, with overwhelming probability,
	\begin{align}\label{eq: event T}
	\text{max}_{1\le i,j\le N} |T_{ij}|\le \frac{5k'}{N}
	\end{align}
	where $k'=\max\left(k, N(\log{N})^{2}\right)$ (put $u=k'/N$).
	Let $\mathcal{T}$ be the event that \eqref{eq: event T} holds. Denote by $\mathcal{E}_{t}\in\FF_{t}$ the event that $\mathcal{T}$ occurs and the conclusion of Lemma \ref{lem: lem10 in BLZ} holds for $R^{[t]}$ (with the convention $A^{[0]}=A$). If $\mathcal{E}_{t}$ holds, then for all $z=E+i\eta$ with $|E-2|\le L^{C_{0}}N^{-2/3}$ and $\eta=L^{-C_{1}}N^{-2/3}$, we have for some $C>0$,
	\begin{align}\label{eq: resolvent bound 1}
	\max_{i\neq j}\left| R_{ij}^{[t]}(z) \right|\le L^{C}N^{-1/3}\quad\text{and}\quad
	\max_{i}\left| R_{ii}^{[t]}(z)\right| \le C
	\end{align}
	We define $A_{0}^{[t]}$ as the symmetric matrix obtained from $A^{[t]}$ by setting to $0$ the $(i_{t}j_{t})$-entry and $(j_{t}i_{t})$-entry. By construction, $A_{0}^{[t+1]}$ is $\FF_{t}$-measurable. We denote by $R_{0}^{[t]}$ the resolvent of $A_{0}^{[t]}$. Recall the following resolvent identity.
	\begin{align}\label{eq: resolvent identity}
	(X-zI)^{-1}=(Y-zI)^{-1}+(Y-zI)^{-1}(Y-X)(X-zI)^{-1}
	\end{align}
	Using the resolvent identity \eqref{eq: resolvent identity}, we get
	\begin{align}
	R_{0}^{[t+1]} = \sum_{l=0}^{2}\left(R^{[t]}(A^{[t]}-A_{0}^{[t+1]})\right)^{l}R^{[t]}
	+ \left(R^{[t]}(A^{[t]}-A_{0}^{[t+1]})\right)^{3}R_{0}^{[t+1]}.
	\end{align}
	We set for $i\le j$, $E_{ij}=\ee_{i}\ee_{j}^{T}+\ee_{j}\ee_{i}^{T}\indic(i\neq j)$ where $\ee_{i}$ denotes the canonical basis of $\R^{n}$ such that the $i$-th entry is equal to $1$ and the other entries are equal to $0$. We have
	\begin{align}
	A^{[t]}-A^{[t+1]}_{0}=a_{i_{t+1}j_{t+1}}E_{i_{t+1}j_{t+1}}\quad\text{and}\quad
	A^{[t+1]}-A^{[t+1]}_{0}=a'_{i_{t+1}j_{t+1}}E_{i_{t+1}j_{t+1}}
	\end{align}
	Recall the fact that $|a_{ij}|\le Cq^{-1}$ and $\lVert R_{0}^{[t+1]}\rVert\le\eta^{-1}$. From now on, we work on the event $\mathcal{E}_{t}$ with fixing $z=E+i\eta$ satisfying $|E-2|\le L^{C_{0}}N^{-2/3}$ and $\eta=L^{-C_{1}}N^{-2/3}$.
	For $i\neq j$, even if we assume the worst case $t+1\in T_{ij}$, we have
	\begin{align}
	\left| (R_{0}^{[t+1]})_{ij} \right| &\le L^{C}N^{-1/3} + Cq^{-1} + Cq^{-2} +Cq^{-3}\eta^{-1} \nn\\
	&=L^{C}N^{-1/3} + Cq^{-1} + Cq^{-2} + CL^{C_{1}}q^{-3}N^{2/3} \nn\\
	&\lesssim L^{C}N^{-1/3} 
	\end{align}
	where we use $q\gg N^{1/3}$ in the last inequality. For $i=j$, it is enough to consider
	\begin{align}
	R_{0}^{[t+1]} = R^{[t]} + R^{[t]}(A^{[t]}-A_{0}^{[t+1]})R^{[t]}
	+ \left(R^{[t]}(A^{[t]}-A_{0}^{[t+1]})\right)^{2}R_{0}^{[t+1]}.
	\end{align}
	Then, it follows that
	\begin{align}
	\left| (R_{0}^{[t+1]})_{ii} \right| \le C + Cq^{-1} + C q^{-2} \eta^{-1} \lesssim C. 
	\end{align}
	In sum,
	\begin{align}\label{eq: resolvent bound 2}
	\max_{i\neq j}\left| (R^{[t+1]}_{0})_{ij}^{[t+1]}(z) \right|\le L^{C}N^{-1/3}\quad\text{and}\quad
	\max_{i}\left| (R^{[t+1]}_{0})_{ii}^{[t+1]}(z)\right| \le C
	\end{align}
	Similarly, by the resolvent identity with $R^{[t+1]}$ and $R^{[t+1]}_{0}$,
	\begin{align}
	R^{[t+1]} & = R_{0}^{[t+1]} + R_{0}^{[t+1]}(A_{0}^{[t+1]}-A^{[t+1]})R_{0}^{[t+1]} + \left(R_{0}^{[t+1]}(A_{0}^{[t+1]}-A^{[t+1]})\right)^{2}R^{[t+1]} \\
	&= \sum_{l=0}^{2}\left(R_{0}^{[t+1]}(A_{0}^{[t+1]}-A^{[t+1]})\right)^{l}R_{0}^{[t+1]}
	+ \left(R_{0}^{[t+1]}(A_{0}^{[t+1]}-A^{[t+1]})\right)^{3}R^{[t+1]}, \label{eq: resolvent identity R t+1 R 0 3rd}
	\end{align}
	we observe
	\begin{align}\label{eq: resolvent bound 3}
	\max_{i\neq j}\left| R_{ij}^{[t]}(z) \right|\le L^{C}N^{-1/3}\quad\text{and}\quad
	\max_{i}\left| R_{ii}^{[t]}(z)\right| \le C.
	\end{align}
	%		Using the resolvent identity with $R^{[t+1]}$ and $R_{0}^{[t+1]}$,
	%		\begin{align}
	%		R^{[t+1]} = R_{0}^{[t+1]} + R_{0}^{[t+1]}(A_{0}^{[t+1]}-A^{[t+1]})R^{[t+1]},
	%		\end{align}
	%		we obtain
	%		\begin{align}
	%		\left| R^{[t+1]}_{ij} - (R_{0}^{[t+1]})_{ij} \right| \le &L^{C}N^{-2/3}q^{-1}\indic(t+1\notin T_{ij}) + L^{C}N^{-1/3}q^{-1}\indic(t+1\in T_{ij})\nn\\
	%		&\quad +Cq^{-1}\indic(\{i,j\}=\{i_{t+1},j_{t+1}\}) =: b_{t+1}.
	%		\end{align}
	Since $R_{0}^{[t+1]}\in\FF_{t}$, it follows from \eqref{eq: resolvent identity R t+1 R 0 3rd} that
	\begin{align}\label{eq: expansion 1}
	\E\left[ R^{[t+1]}_{ij} \middle| \FF_{t} \right] &= (R_{0}^{[t+1]})_{ij} - r_{ij}^{[t+1]} \E[a'_{i_{t+1}j_{t+1}}] + s_{ij}^{[t+1]} \E[(a'_{i_{t+1}j_{t+1}})^{2}] \nn\\
	&\quad - \big((R_{0}^{[t+1]}E_{i_{t+1}j_{t+1}})^{3}R^{[t+1]}\big)_{ij}\E[(a'_{i_{t+1}j_{t+1}})^{3}]
	\end{align}
	where $r_{ij}^{[t]}:=(R_{0}^{[t]}E_{i_{t}j_{t}}R_{0}^{[t]})_{ij}$ and $s_{ij}^{[t]}:=\big((R_{0}^{[t]}E_{i_{t}j_{t}})^{2}R_{0}^{[t]}\big)_{ij}$. We observe
	\begin{align}\label{eq: bound for 3rd order}
	\left| \big((R_{0}^{[t+1]}E_{i_{t+1}j_{t+1}})^{3}R^{[t+1]}\big)_{ij}\E[(a'_{i_{t+1}j_{t+1}})^{3}] \right| \le \alpha_{t+1}
	\end{align}
	where 
	\begin{align}
	\alpha_{t}:=L^{C}N^{-2/3}q^{-3}\indic(t\notin T_{ij}) + L^{C}N^{-1/3}q^{-3}\indic(t\in T_{ij}) + Cq^{-3}\indic(\{i,j\}=\{i_{t},j_{t}\})
	\end{align}
	We write the resolvent identity with $R^{[t]}$ and $R_{0}^{[t+1]}$,
	\begin{align}
	R^{[t]} = \sum_{l=0}^{2}\left(R^{[t+1]}_{0}(A^{[t+1]}_{0}-A^{[t]})\right)^{l}R^{[t+1]}_{0}
	+ \left(R^{[t+1]}_{0}(A^{[t]}-A_{0}^{[t+1]})\right)^{3}R^{[t]},
	\end{align}
	and we get
	\begin{align}
	R^{[t]}_{ij}  = (R^{[t+1]}_{0})_{ij} - r_{ij}^{[t+1]} a_{i_{t+1}j_{t+1}} + s_{ij}^{[t+1]} (a_{i_{t+1}j_{t+1}})^{2} - \big((R_{0}^{[t]}E_{i_{t}j_{t}})^{3}R^{[t]}\big)_{ij}(a_{i_{t+1}j_{t+1}})^{3}.
	\end{align}
	In other words,
	\begin{align}
	(R^{[t+1]}_{0})_{ij}  =  R^{[t]}_{ij} + r_{ij}^{[t+1]} a_{i_{t+1}j_{t+1}} - s_{ij}^{[t+1]} (a_{i_{t+1}j_{t+1}})^{2} + \big((R_{0}^{[t]}E_{i_{t}j_{t}})^{3}R^{[t]}\big)_{ij}(a_{i_{t+1}j_{t+1}})^{3}.
	\end{align}
	Note that, as we show in \eqref{eq: bound for 3rd order}, we can find
	\begin{align}
	\left| \big((R_{0}^{[t+1]}E_{i_{t+1}j_{t+1}})^{3}R^{[t+1]}\big)_{ij}(a_{i_{t+1}j_{t+1}})^{3}] \right| \le \alpha_{t+1}.
	\end{align}
	It follows from \eqref{eq: expansion 1} that
	\begin{align}
	\E\left[ R^{[t+1]}_{ij} \middle| \FF_{t} \right] - R^{[t]}_{ij} &=  r_{ij}^{[t+1]} (a_{i_{t+1}j_{t+1}}-\E[a_{i_{t+1}j_{t+1}}]) - s_{ij}^{[t+1]} \left((a_{i_{t+1}j_{t+1}})^{2}-\E[(a_{i_{t+1}j_{t+1}})^{2}]\right) \nn\\
	&\quad  + \big((R_{0}^{[t]}E_{i_{t}j_{t}})^{3}R^{[t]}\big)_{ij}(a_{i_{t+1}j_{t+1}})^{3}
	-\big((R_{0}^{[t]}E_{i_{t}j_{t}})^{3}R^{[t+1]}\big)_{ij}\E[(a_{i_{t+1}j_{t+1}})^{3}]
	\end{align}
	%(since $a_{ij}\overset{d}{=}a'_{ij}$, we replace $a'_{ij}$ in the expectation with $a_{ij}$). 
	Now we are ready to make the desired bound.
	\begin{align}\label{eq: error terms}
	R^{[k]}_{ij}-R_{ij} = \sum_{t=0}^{k-1}\left( R^{[t+1]}_{ij}-R^{[t]}_{ij} \right) = 
	\sum_{t=0}^{k-1}\left( R^{[t+1]}_{ij}- \E\left[ R^{[t+1]}_{ij} \middle| \FF_{t} \right] \right)
	+ r_{ij} - s_{ij} + \alpha_{ij}
	\end{align}
	where
	\begin{align}
	r_{ij}:=\sum_{t=1}^{k} r_{ij}^{[t]}(a_{i_{t}j_{t}}-\E[a_{i_{t}j_{t}}]),\quad
	s_{ij}:=\sum_{t=1}^{k} s_{ij}^{[t]} \left((a_{i_{t}j_{t}})^{2}-\E[(a_{i_{t}j_{t}})^{2}]\right),\quad
	|\alpha_{ij}|\le 2\sum_{t=1}^{k} \alpha_{t}.
	\end{align}
	We define 
	\begin{align}
	Z_{t+1}:=\left(R^{[t+1]}_{ij}- \E\left[ R^{[t+1]}_{ij}\middle| \FF_{t} \right]\right)\indic(\mathcal{E}_{t}).
	\end{align}
	For any $u\ge 0$, we have
	\begin{align}\label{eq: azuma ineq use}
	\prob\left( \left|\sum_{t=0}^{k-1}\left( R^{[t+1]}_{ij}- \E\left[ R^{[t+1]}_{ij} \middle| \FF_{t} \right] \right)\right| \ge u \right) &\le \prob\left( \left|\sum_{t=1}^{k}Z_{t}\right| \ge u \; ; \; \bigcap_{t=0}^{k}\mathcal{E}_{t}\right) + \prob\left( \bigcup_{t=0}^{k}(\mathcal{E}_{t})^{c} \right) \nn\\
	&\le \prob\left( \left|\sum_{t=1}^{k}Z_{t}\right| \ge u \right) + \sum_{t=0}^{k-1}\prob\big( (\mathcal{E}_{t})^{c} \big). 
	\end{align}
	Since the event $\mathcal{T}$ and the conclusion of Lemma \ref{lem: lem10 in BLZ} hold with overwhelming probability, it is verified that for any $C>0$
	\begin{align}
	\sum_{t=0}^{k-1}\prob\big( (\mathcal{E}_{t})^{c} \big) = O(N^{-C}).
	\end{align}
	Observe that $Z_{s}\in\FF_{t}$ for all $s\le t$,
	\begin{align}
	\E[Z_{t+1}|\FF_{t}] = \indic(\mathcal{E}_{t}) \E\left[\left(R^{[t+1]}_{ij}- \E\left[ R^{[t+1]}_{ij}\middle| \FF_{t} \right]\right) \middle| \FF_{t} \right] = 0,
	\end{align}
	and, by the resolvent identity,
	\begin{align}
	|Z_{t}|\le \beta_{t}
	\end{align}
	where
	\begin{align}
	\beta_{t}:=L^{C}N^{-\frac{2}{3}}q^{-1}\indic(t\notin T_{ij}) + L^{C}N^{-\frac{1}{3}}q^{-1}\indic(t\in T_{ij})+Cq^{-1}\indic(\{i,j\}=\{i_{t},j_{t}\}).
	\end{align}
	Next, we apply Azuma-Hoeffding inequality to bound $\sum_{t=1}^{k}Z_{t}$ because each $Z_{t}$ has zero conditional expectation and is bounded by $\beta_{t}$. It follows that
	\begin{align}
	\prob\left( \left|\sum_{t=1}^{k}Z_{t}\right| \ge  u  \cdot \big({\textstyle\sum}_{t=1}^{k}\beta_{t}^{2}\big)^{1/2}\right)\le 2\exp\left(-\frac{u^{2}}{2}\right).
	\end{align}
	By choosing $u=\log{N}$, we obtain that with overwhelming probability
	\begin{align}\label{eq: control 1st error term}
	\left|\sum_{t=0}^{k-1}\left( R^{[t+1]}_{ij}- \E\left[ R^{[t+1]}_{ij} \middle| \FF_{t} \right] \right)\right| \le (\log{N}) \cdot \big({\textstyle\sum}_{t=1}^{k}\beta_{t}^{2}\big)^{1/2}.
	\end{align}
	Now it is enough to show $\big({\textstyle\sum}_{t=1}^{k}\beta_{t}^{2}\big)^{1/2}\le L^{-C}N^{-1/3}$. Note that
	\begin{align}
	\left(\sum_{t=1}^{k}\beta_{t}^{2}\right)^{1/2}\le L^{C}N^{-2/3}q^{-1}\sqrt{k}+L^{C}N^{-\frac{1}{3}}q^{-1}\sqrt{\frac{k'}{N}}+Cq^{-1}.
	\end{align}
	Let $q=N^{b}$ with $b>1/3$ and $k=N^{\tau}$. Then, it is required that
	\begin{align}
	\max\left( -\frac{2}{3}-b+\frac{\tau}{2}, -\frac{5}{6}-b+\frac{\tau}{2},-\frac{1}{3}-b,-b \right) < -\frac{1}{3}
	\end{align}
	and it boils down to the relation
	\begin{align}\label{eq: relation of exponents}
	\tau < 2b+\frac{2}{3},
	\end{align}
	which is the given assumption. From this relation, we can realize that $\tau < \frac{5}{3}$ due to the assumption $b<\frac{1}{2}$.\\
	In order to bound $r_{ij}$ and $s_{ij}$, we introduce a backward filtration $\{\FF'_{t}\}$. Define $\FF'_{t}$ as the $\sigma$-algebra generated by $A'$, $S_{k}$ and $\{a_{ij}\}_{1\le i,j\le N}\backslash\{a_{i_{s}j_{s}}\}_{1\le s\le t}$. Since $A^{[t]}\in\FF'_{t}$, the event $\mathcal{E}_{t}\in\FF'_{t}$ so the bounds of the resolvents we found above, \eqref{eq: resolvent bound 1}, \eqref{eq: resolvent bound 2} and \eqref{eq: resolvent bound 3}, is still valid with overwhelming probability. Thus, on the event $\mathcal{E}_{t}$, we have
	\begin{align}
	\left| r_{ij}^{[t]}(a_{i_{t}j_{t}}-\E[a_{i_{t}j_{t}}]) \right| \le \beta_{t}.
	\end{align}
	Define $Z'_{t}:=r_{ij}^{[t]}(a_{i_{t}j_{t}}-\E[a_{i_{t}j_{t}}])\indic(\mathcal{E}_{t})$. Note that $Z'_{s}\in \FF'_{t}$ for $s > t$. Since $A^{[t]}_{0}\in\FF'_{t}$, we observe $r_{ij}^{[t]}\in\FF'_{t}$, $\E[Z'_{t} | \FF'_{t} ] = 0$ and $|Z'_{t}|\le\beta_{t}$. According to the same argument we used in \eqref{eq: azuma ineq use}, it is enough to bound
	\begin{align}
	\prob\left( \left|\sum_{t=1}^{k}Z'_{t}\right| \ge u \cdot \big({\textstyle\sum}_{t=1}^{k}\beta_{t}^{2}\big)^{1/2}\right)
	\end{align}
	By applying Azuma-Hoeffding inequality with $u=\log{N}$, with overwhelming probability, it follows that
	\begin{align}
	|r_{ij}|\le (\log{N}) \cdot \big({\textstyle\sum}_{t=1}^{k}\beta_{t}^{2}\big)^{1/2} \le L^{-C} N^{-1/3}.
	\end{align}
	Similarly, we can bound $|s_{ij}|$ using the backward filtration and Azuma-Hoeffding inequality. Since we have
	\begin{align}
	\left| s_{ij}^{[t]} \left((a_{i_{t}j_{t}})^{2}-\E[(a_{i_{t}j_{t}})^{2}]\right) \right| \le q^{-1}\beta_{t},
	\end{align}
	it follows that
	\begin{align}
	|s_{ij}|\le L^{-C}N^{-1/3}.
	\end{align}
	Moreover, we can check
	\begin{align}
	\alpha_{t}\le q^{-2}\beta_{t},
	\end{align}
	which implies, under the relation \eqref{eq: relation of exponents}, 
	\begin{align}
	\alpha_{ij}\le 2\sum_{t=1}^{k}\alpha_{t}\le 2q^{-2}\sum_{t=1}^{k}\beta_{t}\le L^{-C}N^{-1/3}.
	\end{align}
	As a final step, the desired result follows by a continuity argument. Since
	\begin{align}
	\left| R_{ij}(E+i\eta) - R_{ij}(E'+i\eta) \right| \le \frac{|E-E'|}{\eta^{2}},
	\end{align}
	we split the interval $\{E:|E-2|\le L^{C_{0}}N^{-2/3}\}$ into many subintervals of length less than $\eta^{\frac{5}{2}}$. We have, with overwhelming probability, for every end point $w$ of the subintervals,  
	\begin{align}
	\max_{1\le i,j\le N} N\eta|R^{[k]}(w)_{ij}-R(w)_{ij}|\le \frac{1}{L^{2}}.
	\end{align}
	By the continuity, the desired result also hold for all points inside every subintervals..
\end{proof}

\subsection{Proof of Lemma \ref{lem: lem14 in BLZ}}
We write $\vv_{m}=(v_{m,1},\cdots,v_{m,N})$ and $\vv^{[k]}_{m}=(v^{[k]}_{m,1},\cdots, v^{[k]}_{m,N})$ for $m=1,\cdots,N$. By the spectral theorem, we have
\begin{align}
N\eta\text{Im}R(z)_{ij} = \sum_{m=1}\frac{N\eta^{2}v_{m,i}v_{m,j}}{(\lambda_{m}-E)^{2}+\eta^{2}}.
\end{align}
According to Lemma \ref{lem: eig val location} and Lemma \ref{lem: delocalization}, for some $c_{0}>0$, it follows that with overwhelming probability
\begin{align}
|\lambda_{2}- 2|\le L^{c_{0}}N^{-2/3},\quad \lVert \vv_{m} \rVert_{\infty}^{2}\le L^{D}N^{-1}
\end{align}
and, for $m>N':=\lfloor L^{3c_{0}} \rfloor$ and $E$ satisfying $|E-2|\le L^{c_{0}}N^{-2/3}$,
\begin{align}
E-\lambda_{m}\ge c m^{2/3}N^{-2/3}.
\end{align}
Thus, we obtain with overwhelming probability that
\begin{align}
\left|\sum_{m=N'+1}^{N}\frac{Nv_{m,i}v_{m,j}}{(\lambda_{m}-E)^{2}+\eta^{2}}\right| \le L^{D}(N')^{-1/3}N^{4/3}
\end{align}
Fix $\eps>0$. Applying \cite[Theorem 2.7]{EKYY12}, we can find $\delta>0$ such that 
\begin{align}\label{eq: gap between extreme eigevalues}
\prob\left(\lambda_{2}-\lambda_{3} > \delta N^{-2/3}\right)\ge 1-\eps.
\end{align}
As a result, if $|\lambda_{2}-E|\le (\delta/4)N^{-2/3}$, the following event
\begin{align}
\left|\sum_{m=3}^{N'}\frac{Nv_{m,i}v_{m,j}}{(\lambda_{m}-E)^{2}+\eta^{2}}\right| 
\le \frac{4L^{D}N'N^{4/3}}{\delta^{2}}
\end{align}
holds with probability at least $1-\eps$. If $|\lambda_{2}-E|\le \eta L^{-C_{1}}$, Lemma \ref{lem: delocalization} implies that
\begin{align}
\left| \frac{N\eta^{2}v_{2,i}v_{2,j}}{(\lambda_{2}-E)^{2}+\eta^{2}} - Nv_{2,i}v_{2,j}\right|\le L^{D-2C_{1}}
\end{align}
holds with overwhelming probability.
Also, since $\lambda_{1}\sim \zeta q+1/\zeta q$ with overwhelming probability by Lemma \ref{lem: top eig val location}, we can verify that
\begin{align}
\left|  \frac{N\eta^{2}v_{2,i}v_{2,j}}{(\lambda_{1}-E)^{2}+\eta^{2}} \right| \le \frac{L^{D-2C_{1}}N^{-4/3}}{q^{2}}
\end{align}
holds with overwhelming probability (Lemma \ref{lem: delocalization} is also used). Combining all of the above estimates and choosing $C_{1}>0$ large enough, we conclude that for all $E$ satisfying $|\lambda_{2}-E|\le \eta L^{-C_{1}}$
\begin{align}
\max_{1\le i\le j\le N} |N\eta\text{Im}R(E+i\eta)_{ij}-Nv_{2,i}v_{2,j}|\le L^{-2}
\end{align}
with probability at least $1-\eps$. Now we repeat the above all argument with replacing $R$ with $R^{[k]}$, which provides us that for all $E$ satisfying $|\lambda_{2}^{[k]}-E|\le \eta L^{-C_{1}}$
\begin{align}
\max_{1\le i\le j\le N} |N\eta\text{Im}R^{[k]}(E+i\eta)_{ij}-Nv^{[k]}_{2,i}v^{[k]}_{2,j}|\le L^{-2}
\end{align}
with probability at least $1-\eps$.
According to Lemma \ref{lem: lem12 in BLZ}, we have $ |\lambda_{2}-\lambda_{2}^{[k]}|\le \eta L^{-C_{1}}$ with overwhelming probability. Therefore $\eqref{eq: difference between resolvent and vector components}$ holds with probability at least $1-3\eps$.

\end{appendices}

\end{document}